\documentclass[12pt]{amsart}
\usepackage{amssymb}
\usepackage[dvips]{graphicx}
\usepackage[latin1]{inputenc}
\usepackage{young}
\usepackage{wasysym}
\usepackage[usenames]{color}
\numberwithin{equation}{section}
\newtheorem{thm}{Theorem}
\newtheorem{prop}[thm]{Proposition}
\newtheorem{lemma}[thm]{Lemma}
\newtheorem{cor}[thm]{Corollary}

\theoremstyle{definition}
\newtheorem{example}[thm]{Example}
\newtheorem{remark}[thm]{Remark}
\newtheorem{openproblem1}[thm]{Open Problem}
\newtheorem{definition}[thm]{Definition}
\newtheorem{assum}[thm]{Assumption}
\newenvironment{ex}{\begin{example}\rm}{\end{example}}

\numberwithin{thm}{section}
\newcommand{\p}{{\rm P}}
\def\l{{\rm L}}
\def\u{{\rm U}}
\numberwithin{thm}{section}
\newcounter{FNC}[page]
\def\newfootnote#1{{\addtocounter{FNC}{2}$^\fnsymbol{FNC}$%
    \let\thefootnote\relax\footnotetext{$^\fnsymbol{FNC}$#1}}}
\newcommand{\C}{\mathbb{C}}
\newcommand{\K}{\mathbb{K}}
\newcommand{\N}{\mathbb{N}}

\newcommand{\CC}{\mathcal{C}}
\newcommand{\R}{\mathbb{R}}

\newcommand{\Z}{\mathbb{Z}}

\definecolor{RED}{rgb}{0.6,0,0}

\DeclareMathOperator{\id}{\mathrm{id}}

\newcommand{\atopfrac}[2]{\genfrac{}{}{0pt}{}{#1}{#2}}

\DeclareMathOperator{\Gl}{GL}

\DeclareMathOperator{\sgn}{sgn}

\DeclareMathOperator{\wt}{wt}
\DeclareMathOperator{\Van}{Van}
\DeclareMathOperator{\CStab}{CStab}

\DeclareMathOperator{\myspan}{span}

\author{Cordian Riener}
\address{Fachbereich Mathematik und Statistik,
Universit\"at Konstanz, 78457 Konstanz, Germany.}
\email{cordian.riener@uni-konstanz.de}
\author{Thorsten Theobald}
\address{FB 12 -- Institut f\"ur Mathematik,
Goethe-Universit\"at,
Postfach 11 19 32, 60054 Frankfurt am Main, Germany.}
\email{theobald@math.uni-frankfurt.de}
\author{Lina Jansson Andr\'en}
\address{
Dept.\ of Mathematics and Math.\ Statistics, Ume\aa{} universitet,
901 87 Ume\aa{}, Sweden.}
\email{lina.andren@math.umu.se}
\author{Jean B.\ Lasserre}
\address{LAAS-CNRS and Institute of Mathematics, LAAS 7 Avenue du Colonel Roche, 31077 Toulouse Cedex 4, France.}
\email{lasserre@laas.fr}

\title{Exploiting symmetries in SDP-relaxations for polynomial optimization}
\thanks{An earlier preprint version of this paper already received attention and
is referenced e.g. in the surveys \cite{bgsv-2010,laurent-2009,vallentin-2007}.
}

\begin{document}

\begin{abstract}
In this paper we study various approaches for
exploiting symmetries in polynomial optimization problems
within the framework of semidefinite programming relaxations.
Our special focus is on constrained problems especially when
the symmetric group is acting on the variables.
In particular, we investigate the concept of block decomposition within
the framework of constrained polynomial optimization problems,
show how the degree principle for the symmetric group can be
computationally exploited and also propose some methods to efficiently
compute in the geometric quotient.
\end{abstract}

\maketitle
\section{Introduction}

Solving or even computing lower bounds in constrained polynomial optimization
is a difficult problem with important practical applications.
In recent years, results of real algebraic geometry on the representation of positive polynomials
have permitted to define a hierarchy of semidefinite relaxations (SDP-relaxations) of these problems,
which provide a monotonically nondecreasing sequence of lower bounds converging to the global minimum.
See e.g.\ Lasserre \cite{lasserre-2001}, Parrilo \cite{parrilo-2001} or the survey \cite{laurent-2009} and many references therein. However, the size of the resulting SDPs grows fast with the problem size; typically, for an optimization problem in $n$ variables the SDP-relaxation of order $k$ in the hierarchy involves $O(n^{2k})$ variables and linear matrix inequalities (LMIs) of size $O(n^k)$. Therefore, and in view of the present status of SDP solvers, the applicability of the basic methodology is limited to small or medium size problems unless some specific characteristics are taken into account.

One way to reduce this size limitation is to exploit \emph{symmetries} when present in the problem
definition. In the present paper, which has a foundational character,
we consider the polynomial optimization problem
\begin{equation}
\label{eq:opt1}
 \begin{array}{rcll}
 f^* & = & \inf f(x) \\
 & & \text{s.t. } g_1(x) \ \ge \ 0, \ldots, g_m(x) \ \ge \ 0 \, ,
 \end{array}
\end{equation}
where $f,g_1, \ldots, g_m \in \R[X_1, \ldots, X_n]$.
We assume that the polynomials are invariant by the action of a finite subgroup $G$ of the group $\Gl_n(\R)$,
i.e., $f(\sigma^{-1}(x)) = f(x)$
and $g_j(\sigma^{-1}(x)) = g_j(x)$ for all $\sigma \in G$, and all $j=1, \ldots,m$.

A major theoretical contribution to a systematic study of symmetries in real algebraic
geometry was provided by Procesi and Schwarz \cite{proschwa} who gave a semi-algebraic
description of the geometric quotient of a semi-algebraic set invariant under
a group $G$ \cite{broecker}. For the special case of the symmetric group,
Timofte \cite{timofte-2003} provided a very useful criterion for the
non-negativity of a polynomial.
Cimpri\v{c}, Kuhlmann, and Scheiderer studied foundational aspects of the dual
problem of moments \cite{cimpric}.

The systematic study of block diagonalizations of SDPs was initiated by
Gatermann and Parrilo \cite{gatermann-parrilo-2004} (in the context of symmetries) and by
Schrijver \cite{Schrijver2,Schrijver} (in the general framework of matrix $*$-algebras).
Building on this, de Klerk, Pasechnik, and Schrijver \cite{dps-2006} have provided a general
method (the $*$-representation) to handle symmetries of any semidefinite program (see
also \cite{gijswijt-2005,komk-2001,laurent-2006}).
Excellent reviews of the above are the surveys
\cite{bgsv-2010,vallentin-2007}.

\smallskip

$*$ {\bf Contribution:} In the present paper, we advance these lines of research in several ways:

1. We provide a systematic treatment of the block diagonalization in
the setting of Lasserre's relaxation which is concerned with \emph{constrained}
optimization. Instead of considering
a general SDP framework, we rather focus attention on
the specific SDPs coming from the relaxation scheme defined in
\cite{lasserre-2001}. Indeed,
the symmetries on the original variables of the optimization problem
induce specific additional symmetry structure on the moment and
localizing matrices of the SDP-relaxation.
To this end we suggest that a symmetry-adapted version of the relaxation scheme
can be defined directly using an appropriate basis for the moments and derive symmetric versions of Putinar's Theorems (see Theorems \ref{eq:SymMom} and \ref{thm:symPutinar}).
We study a possible basis (generalized Specht polynomials as
defined in Section~\ref{se:momentmatrixsymmetric})
in detail for the case of the symmetric group
$\mathcal{S}_n$.
In this situation we show that for $k$ fixed,
the number and sizes of the LMIs in the SDP-relaxation of order $k$ are
bounded by a constant that does {\it not}
depend on the number $n$ of variables (Theorem~\ref{th:constantasym}).
As a direct consequence, we can state some symmetric versions of
representation theorems
for sums of squares, in particular for the ``Hilbert cases''
(Theorem \ref{thm:darst} and corollaries).

2. We show how the so-called degree principle (\cite{Rie,timofte-2003}) can be used to transform an $\mathcal{S}_n$-invariant optimization problem into a set of lower dimensional problems and that in some cases the
resulting relaxation scheme converges finitely (Theorem~\ref{th:finiteconvergence}).
This gives a sum of squares based criterion to
certify non-negativity of an $\mathcal{S}_n$-symmetric polynomial of degree~4
(Theorem~\ref{thm:degfour}).

3. We show how the geometric quotient viewpoint naturally leads to a Polynomial Matrix Inequality (PMI) problem. For certain power sum problems
(generalizing a situation studied by Brandenberg and
Theobald \cite{brandenberg-theobald-2006}), we discuss how this leads
to lower and upper bounds which can be computed quite simply (Theorems~\ref{lower1}
and~\ref{upper}).

\smallskip

Our techniques enlarge the techniques for handling constrained optimization
problems with symmetries. We feel that it is worth to present them in a
common context. We focus to a large extent on the case of the
group $\mathcal{S}_n$.  This has several reasons.
Firstly the problems that motivated the research leading to this paper came from this setting.
Secondly it turns out that in the situation of symmetric polynomials the complexity of the optimization problem as a function of the number of variables can be dramatically reduced with all the techniques we provide.
Moreover, the symmetric group serves as a rich example to demonstrate
general principles, and we remark that many combinatorial optimization
problems can be put into the form of maximizing a given linear
form on the orbit of a vector in a representation of the symmetric group
(see \cite{barvinok-vershik-89,barvinok-92}).
Whereas (on the SDP level) the general framework of block diagonalization is already well understood the degree principle still awaits its generalization for other groups.

The authors are aware that certain algebraic techniques used in this paper might not
be very familiar to optimizers in general, which may induce
doubts about any real systematic implementation
in some fully automized software, at least in some near future.
However, there are also reasons to be more optimistic
in view of the growing interest in semidefinite relaxations for polynomial optimization,
and their current limitation to problems of modest size only, if no sparsity or symmetry is
not taken into account.

\medskip

The paper is structured as follows. In Section~\ref{se:pre}, we
give a short introduction to the SDP relaxation scheme and to
PMIs.
Furthermore we introduce some representation theoretical notions,
with special focus on the symmetric group $\mathcal{S}_n$.
In Section \ref{se:symadapt} we give a systematic treatment of how invariance by a finite group can be exploited in the relaxation scheme introduced in \cite{lasserre-2001}.
Section \ref{se:sympol} is devoted to a study of optimization with symmetric polynomials. We give a detailed construction of the related moment matrices. From the constructions we then deduce representation statements for symmetric positive polynomials. 
In Section~\ref{se:timofte} we show how it is possible to use the degree principle to break some of the symmetry and construct thereby a family of lower
dimensional problems, which can be used to solve the original optimization
problem.

Finally, in Section~\ref{se:pmirelaxations}
we show how optimization problems described by invariant polynomials
can be treated in the orbit space.
As a direct application of this procedure we can show how to calculate bounds
for a specific class of problems.

\section{Preliminaries}\label{se:pre}
Let $\R[X]$ be the ring of polynomials in the variables $X=(X_1,\ldots,X_n)$
and $\R[X]_{\le k}$ be the subset of polynomials of degree at most $k$.
In the following subsections, we recall Lasserre's relaxation
scheme for polynomial optimization, polynomial matrix inequalities (PMIs)
and some basic concepts of representation theory.

\subsection{Lasserre's method\label{se:lasserre}}
Given polynomials $f, g_1, \ldots, g_m  \in \R[X]$, consider the general
optimization problem of the form
\[
f^* \ = \ \inf {f(x)} \  \text{ subject to } \ g_1(x)\geq 0,\ldots, g_m(x) \geq 0 \, .
\]
Its feasible set $K\subseteq\R^n$ is the basic closed semi algebraic set
\begin{equation}\label{setk}
K \ := \ \{x\in\R^n\::\: g_j(x)\geq0, \; j=1,\ldots,m\}.
\end{equation}

In \cite{lasserre-2001} Lasserre has introduced the following hierarchy of
semidefinite relaxations (see also \cite{lassere-book,laurent-2009}).
For reasons described below we will need the following technical assumption:
\begin{assum}
\label{putinar-assumption}
The feasible set $K$ defined in (\ref{setk})
is compact and there exists a polynomial $u\in\R[X]$ such that
the level set $\{x\in\R^n\::\:u(x)\geq0\}$ is compact and $u$ has the
representation
\begin{equation}\label{put}
u \ = \ u_0+\sum_{j=1}^m u_j\,g_j
\end{equation}
for some sums of squares polynomials $u_0,u_1, \ldots, u_m \in \R[X]$.
\end{assum}
Assumption \ref{putinar-assumption} holds if e.g.\ for some
$j\in\{1,\ldots,m\}$ the level set
$\{x\in\R^n\::\:g_j(x)\geq 0\}$ is compact, or if
$K$ is compact and all the $g_j$'s are affine (in which case $K$ is a
polytope).
In particular, Assumption \ref{putinar-assumption} holds if and only if for some
$N\in\N$, the polynomial
$N - \sum_{i=1}^n X_i^2$ can be written in the form
(\ref{put});
equivalently, this polynomial belongs to the quadratic module generated by the
$g_j$'s. For a comprehensive discussion of the
last condition see~\cite[Theorem 1]{schweighofer-2005}.
Notice that under Assumption \ref{putinar-assumption}, $K$ is compact and thus 
the infimum $f^*$ is attained on $K$.

The idea is to convexify the problem by considering the equivalent formulation
\begin{equation}
\label{eq:moment1}
f^* \ = \ \min_{x \in K} f(x) \ = \ \min_{\mu \in \mathcal{P}(K)}
\int f \, d\mu \, ,
\end{equation}
where $\mathcal{P}(K)$ denotes the set of all probability measures
$\mu$ supported on the set $K$. These measures are characterized
by the following statement due to Putinar.

\begin{thm}[Putinar \cite{put}]\label{th:putinar}
Suppose Assumption~\ref{putinar-assumption} holds for the set $K$. A linear map $L:\,\R[X]\to \R$ is the
integration with respect to a probability measure $\mu$ on $K$, i.e.,
$$\exists\mu \in \mathcal{P}(K) \quad \forall p\in\R[X] \quad \,L(p) \ = \ \int p \, d\mu \, ,$$ if and only if $L(1)=1$
and $L(s_0+ \sum_{j=1}^m s_j g_j) \geq 0$ for any 
sum of squares polynomials $s_0, \ldots, s_m \in \R[X]$.
\end{thm}
Setting $g_0 := 1$,
the condition in Putinar's result is satisfied if and only if 
the bilinear forms 
$\mathcal{L}_{g_{0}}, \ldots, \mathcal{L}_{g_{m}}$ defined by
\begin{eqnarray*}
\mathcal{L}_{g_{j}}:\R[X]\times \R[X]&\to&\R, \\
(p,q)&\mapsto&L(p\cdot q\cdot g_j)
\end{eqnarray*}
are positive semidefinite (psd).
With this characterization we can restate \eqref{eq:moment1} as
\begin{equation}\label{eq:p1}
f^{*}  \ = \ \min\left\{ L(f) \; : \; L:\R[X] \to \R \mbox{ linear},\, L(1)=1 \text{ and each } \mathcal{L}_{g_j} \text{ is  psd}\right\} \, .
\end{equation}
Now fix any basis $\mathcal{B}$ of the vector space $\R[X]$ with 
$1 \in \mathcal{B}$ (for example the monomial basis $X^\alpha$).
For any linear map $L:\R[X] \to \R$ with $L(1) = 1$, 
setting $y_{b} = L(b)$ for $b \in \mathcal{B}$ identifies $L$ with
an infinite series $y = (y_b)_{b \in \mathcal{B}}$ 
of real numbers indexed by the elements of~$\mathcal{B}$.
The infinite-dimensional \emph{moment matrix} $M$ associated to $y$ is 
indexed by $\mathcal{B}$ and given
by $$M(y)_{u,v} \ := \ L(u \cdot v) \, , \quad u, v \in \mathcal{B} \, .$$
Furthermore for each $g_j$ define in an analogous manner the
\emph{localizing matrix} $M(g_j\, y)$ by
$$M(g_j \, y)_{u,v} \ := \ L(u \cdot v \cdot g_j), \quad u, v \in \mathcal{B} \, .
$$
Under Assumption~\ref{putinar-assumption},
a given sequence $y$ comes from some measure $\mu$ supported on $K$ if and only if the 
moment matrix as well as the localizing matrices are psd.
For practical applications of this approach, truncated versions 
of~\eqref{eq:p1} have to be considered: 
Let $k\geq k_0 := \max  \{ \lceil \deg f/2 \rceil, \lceil \deg g_1/2
\rceil, \ldots, \lceil \deg g_m/2 \rceil \}$. Define
the finite-dimensional matrix $M_k$ by 
considering only rows and columns indexed by elements in 
$\mathcal{B}$ of degree at most~$k$,
and consider the hierarchy of semidefinite relaxations
\begin{equation}\label{relax}
 Q_k:\quad\begin{array}{rcl}
 \inf_y L(f) \\
 M_k(y) & \succeq & 0 \, , \\
 M_{k - \lceil \deg g_j / 2 \rceil}(g_j \, y) & \succeq & 0 \, , \quad
1 \le j \le m \, , \\
  y_1 & = & 1 \, ,
 \end{array}
\end{equation}
with optimal value denoted by $\inf Q_k$ (and $\min Q_k$ if the infimum
is attained).

Although each of the relaxation values might not be optimal
for the original problem, one has the following convergence result.
\begin{prop}[Lasserre \cite{lasserre-2001}] \label{pr:lasserreconv}
Let Assumption \ref{putinar-assumption} hold and consider the hierarchy
of SDP-relaxations $(Q_k)_{k \ge k_0}$ defined in (\ref{relax}).
Then the sequence $(\inf Q_k)_{k\geq k_0}$ is monotonically  non-decreasing and
converges to $f^*$; that is,
$\inf Q_k \uparrow f^*$ as $k\to\infty$.
\end{prop}

Although there are sufficient conditions to decide whether an optimal value has been reached after a certain iteration (see for example \cite{detection,lassere-book}), in general only in some situations finite convergence can be guaranteed: 

\begin{prop}[Laurent \cite{laurent-2007}] \label{prop:zerodim}
Let $f,g_1,\ldots, g_m\in\R[X]$ and consider the problem
\begin{equation}
\label{eq:laurentcase}
\inf\limits_{x\in \R^n} \{ f(x) \, : \, g_1(x)=\cdots=g_m(x)=0 \} \, .
\end{equation}
If the ideal generated by $g_1, \ldots, g_m$ is zero-dimensional then 
the Lasserre relaxation scheme of~\eqref{eq:laurentcase} 
has finite convergence, i.e., there is an $l\geq k_0$ such that 
$\inf Q_l =f^*$.
\end{prop}

\subsection{Polynomial matrix inequalities\label{se:pmi}}

An interesting case of a polynomial optimization problem which will be
relevant for some of our approaches arises
when dealing with positive semidefiniteness of a matrix whose entries are polynomials.

Let $S_m$ denote the set of real symmetric $m \times m$-matrices.
A \emph{polynomial matrix inequality (PMI)} optimization problem
is an optimization problem of the form
\[
 \begin{array}{rcl}
   f^* & = & \inf_{x\in\R^n} f(x) \\
       &   & \text{s.t.\ } G(x) \succeq 0
 \end{array}
\]
where $f \in \R[X]$ and $G(X)$ is a symmetric $m \times m$-matrix whose entries
$G_{ij}(X)$ are polynomials in $X$.

By considering the psd condition on $G(x)$ as polynomial constraints
and using the approach from Subsection~\ref{se:lasserre},
one would have to deal with polynomials of large degree. Even if all $G_{ij}(X)$ are linear for example the polynomial inequalities one needs to consider are of degree $m$.  This high degree could make it even hard to explicitly calculate the first possible relaxation.

To overcome this problem, SDP hierarchies were proposed
in \cite{henrion-lasserre-pmi,Hol-Scherer} that take into account
the semidefiniteness of a polynomial matrix.
The basic idea is to generalize the standard approach
in a suitable way
by defining a localizing matrix for the matrix $G(X)$. This (infinite)
matrix consists of blocks which are indexed by the elements of a basis $\mathcal{B}$ of $\R[X]$, the entries of each block are indexed with the entries of $G(X)$, and the entry $i,j$ of the block corresponding to $u,v\in\mathcal{B}$ is
$$M(G \, y)^{u,v}_{i,j} \ := \ L(u \cdot v \cdot G_{ij}(X)), \quad u, v \in \mathcal{B} \,,
\quad i,j\in\{1,\ldots,m\} .$$

Setting 
$d:=\max\{ \lceil \deg G_{ij}(X)/2 \rceil\}$ and $k\geq k_0:=\max  \{ \lceil \deg f/2 \rceil, d\}$, 
one can define a relaxation
\begin{equation}\label{relax2}
 Q_k:\quad\begin{array}{rcl}
  \multicolumn{3}{l}{\inf_{y} L(f)} \\
 M_k(y) & \succeq & 0 \, , \\
 M_{k-d}(G \, y) & \succeq & 0 
 \end{array}
\end{equation}
with the truncated matrix $M_k$ of $M$.
In order to guarantee the convergence of this relaxation one needs to assume the Putinar condition viewed in this setting:
\begin{assum}
\label{putinar-assumption2}
Suppose that there is $u\in\R[X]$ such that
the level set $\{x\in\R^n\::\:u(x)\geq0\}$ is compact and $u$ has the
representation
\begin{equation}
\label{put2}
u\,=\,u_0+\langle R(X),G(X)\rangle
\end{equation}
for some sum of squares polynomials $u_0\in\R[X]$ and a symmetric
sum of squares matrix  $R(X)\in\R[X]^{m\times m}$.
\end{assum}

Then the following convergence statement holds \cite{henrion-lasserre-pmi}.

\begin{prop}\label{prop:Las_PIM}
If $G(X)$ meets the Assumption~\ref{putinar-assumption2}
then the sequence $(\inf Q_k)_{k\geq k_0}$ is monotonically non-decreasing and
converges to $f^*$.
\end{prop}

\subsection{Linear representation theory\label{se:representation}}
We collect some notions from linear representation theory. As 
standard reference see \cite{serre-b77}. 
For our purposes, we always assume that $G$ is a finite group and 
that $\K$ is either the field $\R$ of real numbers or the field $\C$ of complex numbers.
 
A \emph{representation of} $G$ is a finite-dimensional vector space $V$ over $\K$
together with a group homomorphism $\rho: G\rightarrow \Gl(V)$ 
into the set $\Gl(V)$ of invertible linear transformations of $V$. 
If a basis for $V$ is chosen, then the representation can be expressed as a 
group homomorphism into the group $\Gl_n(\K)$, where $n := \dim V$. 
This is known as a \emph{matrix representation}.
The action of $G$ turns $V$ into a $G$-\emph{module}, and in fact,
the notion of a representation of $G$ and the notion of a $G$-module are
equivalent and can be identified. 
Two representations $(V,\rho)$ and $(V',\rho')$ of the same group $G$ are \emph{equivalent} if there
is a linear isomorphism $\phi: V\rightarrow V'$ such that
$\rho'(\sigma) \ = \ \phi \, \rho(\sigma) \, \phi^{-1} \quad \text{for all } \sigma\in G.$

\begin{ex}
\label{ex:reprex}
\begin{enumerate}
\item The one-dimensional representation $(\id,\K)$
(i.e., $V=\K$ and group action $\sigma(v)=v$ for all $\sigma\in G$ and $v\in\K$)
is called the \emph{trivial representation}.
\item Take any set $S$ on which $G$ operates and set
$V = \bigoplus_{s\in S}\C e_{s}$ with formal symbols $e_s$ ($s \in S$).
Then the obvious action of $G$ on $V$ defined via $\sigma(e_s)=e_{\sigma(s)}$ 
turns $V$ into a $G$-module. In the special case when $S=G$ this is called the \emph{regular representation}.
\end{enumerate}
\end{ex}

If there is a proper $G$-submodule $W$ of~$V$ (i.e., a $G$-invariant
proper subspace $W$ of $V$)
then the representation $(\rho,V)$ is called
\emph{reducible}. If, however, the only $G$-invariant subspaces are
$V$ and
$\{0\}$, then $(\rho,V)$ is called \emph{irreducible}.
For a finite group $G$, any representation $(V,\rho)$ decomposes
 as a direct sum of irreducible representations,
$\rho \ = \ \bigoplus_{i=1}^h m_i \rho_i$ with multiplicities
$m_i$. This induces an \emph{isotypic decomposition} of
the $G$-module $V$ as a direct sum
$V = \bigoplus_{i=1}^h V_i$
with \emph{isotypic components} $V_i = \bigoplus_{j \in J_i} W_{ij}$,
finite index sets $J_i$ with $|J_i| = m_i$,
and, for fixed $i$, pairwise isomorphic, irreducible components $W_{ij}$ 
(where each component $W_{ij}$
is associated with the irreducible representation $\rho_i$).
Whereas the decomposition in isotypic components is unique, the decomposition of the $V_i$ is in general not unique. 

The following basic but fundamental result of representation theory will be very useful
(see, e.g., \cite[Proposition~4]{serre-b77}).

\begin{lemma}[Schur's Lemma] \label{le:Schur}
Let $(\rho_1,V)$ and $(\rho_2,W)$ be two irreducible representations of a group 
$G$ (over $\R$ or $\C$).
Then every $G$-homomorphism $V \to W$ is either zero or an 
isomorphism.
Moreover, if two complex irreducible representations 
$(\rho_1,V)$ and $(\rho_2,W)$ are isomorphic
then the vector space of all $G$-homomorphisms $V \to W$ has dimension~1.
\end{lemma}

When working with real representations
a little precaution is necessary  and it can be useful to pass to the complexification 
of the real representation. For a real irreducible representation $(\rho, V)$
the following three types are distinguished
(see~\cite[Section~13.2]{serre-b77}):
If the complexification $V\otimes \C$ is also irreducible (type {\bf I}) then 
statements such as the second part of
Lemma~\ref{le:Schur} directly transfer from $V \otimes \C$ to $V$.
However, it may occur that a real irreducible representation $(\rho,V)$ 
becomes reducible when passing to the complexification. 
In this case, $V\otimes\C$ will decompose into two complex-conjugate irreducible
$G$-submodules $V_1$ and $V_2$.
Since $V_1$ and $V_2$ are complex conjugates we can ``virtually'' keep track of this decomposition by decomposing $V\otimes\C$ as 
$V_1+V_2 \oplus \frac{1}{\mathrm{i}}(V_1-V_2)$, where $\mathrm{i}$ is the imaginary unit.
This is a real basis of $V$, which respects the decomposition of $V\otimes \C$.
Either the $G$-submodules $V_1$ and $V_2$ are non-isomorphic (type {\bf II}) 
or they are isomorphic (type {\bf III}).

Let $V$ be a real $G$-module. 
If the complexification $V \otimes C$ has the isotypic decomposition
$V\otimes\C=V_1\oplus\cdots\oplus V_{2l}\oplus V_{2l+1} \oplus\cdots\oplus V_h$, 
where each pair $(V_{2j-1}, V_{2j})$ is complex conjugate ($1 \le j \le l$) 
and $V_{2l+1},\ldots, V_{h}$ are real, the decomposition
\begin{equation}\label{eq:realdecom}
V \ = \ \left( V_1+V_2 \right)
\oplus \frac{1}{\mathrm{i}} \Big(V_1-V_2 \Big)\oplus \cdots \oplus \left( V_{2l-1}+V_{2l} \right) \oplus \frac{1}{\mathrm{i}}
\Big( V_{2l-1} - V_{2l} \Big)\oplus V_{2l+1}\oplus\cdots\oplus V_{h}\end{equation}
is called a \emph{real decomposition}.
     
\subsection{Symmetric group}\label{se:prelimsymmgroup}
An important special case is when $G$ is the symmetric group
$\mathcal{S}_n$ on $n$ variables.
We collect some well-known facts on the irreducible representations of 
$\mathcal{S}_n$.
For a general reference we refer to~\cite{sagan-2001}.

For $n \ge 1$,
a \emph{partition} $\lambda$ of $n$
(written $\lambda\vdash n$)
is a sequence of weakly decreasing positive integers
$\lambda=(\lambda_1,\lambda_2,\ldots,\lambda_l)$ with
$\sum_{i=1}^l\lambda_i=n$. For two partitions $\lambda, \mu \vdash n$
we write $\lambda \unrhd \mu$ if
$\lambda_1 + \cdots + \lambda_i \ge \mu_1 + \cdots + \mu_i$ for all $i$.
A \emph{Young tableau} for $\lambda\vdash n$ consists of $l$ rows, with
$\lambda_i$
entries in the $i$-th row.
Each entry is an element in $\{1, \ldots, n\}$, and each of these
numbers occurs
exactly once.
A \emph{standard Young tableau} is a Young tableau in which all rows and
columns
are increasing.

\begin{ex} For the partition $\lambda=(4,3,1,1,1) \vdash 10$, an example of a
Young tableau is

\hspace{60mm}  \begin{Young}
    1  &  3 &  4 & 6\cr
    5  &  7 & 8   \cr
    9       \cr
    2      \cr
    10           \cr
    \end{Young} \, .

\end{ex}
An element $\sigma \in \mathcal{S}_n$ acts on a Young tableau by
replacing each entry by its image under $\sigma$.
Two Young tableaux $t_1$ and $t_2$ are called \emph{row equivalent}
if the corresponding rows of the two tableaux contain the same numbers.
The classes of equivalent Young tableaux are called \emph{tabloids}, and
the class
of a tableau $t$ is denoted by $\{t\}$.
Let $\{t\}$ be a $\lambda$-tabloid. 
The action of $\mathcal{S}_n$ gives rise to an $\mathcal{S}_n$-module:

\begin{definition}\label{def:Permutationmodule}
Suppose $\lambda\vdash n$.
The \emph{permutation module} $M^{\lambda}$
\emph{corresponding to} $\lambda$ is the $\mathcal{S}_n$-module
defined by  $M^\lambda=\myspan \{ \{t_1\}, \ldots ,\{t_l\}\}$,
where $\{t_1\}, \ldots, \{t_l\}$ is a complete list of
$\lambda$-tabloids.
\end{definition}

\begin{ex}
If $\lambda=(1,1, \ldots ,1) \vdash n$
then $M^\lambda$ is isomorphic to the
regular representation of 
$\mathcal{S}_n$ from Example~\ref{ex:reprex}.
In case $\lambda = (2,1)$ a complete list of $\lambda$-tabloids is
given by the representatives

\begin{equation}
 \label{eq:young1}
 \begin{Young}
    1  &  2 \cr
    3 \cr
 \end{Young} \qquad
 \begin{Young}
    1  &  3 \cr
    2 \cr
 \end{Young} \qquad
 \begin{Young}
    2  &  3 \cr
    1 \cr
 \end{Young} \quad .
\end{equation}
\end{ex}
Let $t$ be a Young tableau for $\lambda\vdash n$, and let 
$\mathcal{C}_1, \ldots, \mathcal{C}_{\nu}$ be the columns of $t$.
The group $\CStab_t \ = \ \mathcal{S}_{\mathcal{C}_1}\times 
\mathcal{S}_{\mathcal{C}_2}\times
\cdots \times \mathcal{S}_{\mathcal{C}_\nu}$
(where $\mathcal{S}_{\mathcal{C}_i}$ is the symmetric group on $\mathcal{C}_i$)
is called the \emph{column stabilizer} of $t$.

The irreducible representations of the symmetric group $\mathcal{S}_n$
are in 1-1-correspondence
with the partitions of $n$, and they are given by the Specht modules, as
explained in the following.

For $\lambda\vdash n$, the \emph{polytabloid associated with} $t$ is
defined by
\begin{equation}
\label{eq:polytabloid}
e_t \ = \ \sum_{\sigma\in \CStab_t} \sgn(\sigma)\sigma\{t\} \, ,
\end{equation}
where $\sgn(\sigma)$ denotes the signum of $\sigma$.
Then for a partition $\lambda\vdash n$, the \emph{Specht module}
$S^{\lambda}$
is the submodule of the permutation module $M^\lambda$ spanned by the
polytabloids $\{e_t \, : \, t \text{ Young tableau for }\lambda \vdash n\}$.
The dimension of $S^{\lambda}$ is given by the number of standard
Young tableaux for $\lambda \vdash n$.

\begin{example}
\label{ex:decomp1}
For $n \ge 2$, we have the decomposition into irreducible components
$M^{(n-1,1)} = S^{(n)} \oplus S^{(n-1,1)}$. Namely,
since the one-dimensional subspace spanned by the sum $t_1 + \cdots +
t_n$ is closed
under the action of $\mathcal{S}_n$, we have a copy of the
\emph{trivial}
representation (which is isomorphic to the Specht module $S^{(n)}$)
as irreducible component in $M^{(n-1
,1)}$. Moreover,
since the tabloids in~\eqref{eq:young1} are completely determined by the
entry in the second row, we have identified a copy of the
$(n-1)$-dimensional Specht module $S^{(n-1,1)}$ in $M^{(n-1,1)}$.
Indeed, the permutation module
$M^{(n-1,1)}$ decomposes as $M^{(n-1,1)} = S^{(n)} \oplus S^{(n-1,1)}$.
\end{example}
The decomposition of the module $M^\lambda$ for a general partition $\lambda\vdash n$ will be of special interest for us.
It can be described in a rather combinatorial way as follows:
\begin{definition}
\begin{enumerate}
\item A generalized \emph{Young tableau} of \emph{shape} $\lambda$ is a Young tableau $T$ for $\lambda$ such that the entries are replaced by any $n$-tuple of natural numbers. The \emph{content} of $T$ is the sequence $\mu$ such that $\mu_i$ is equal to the number of $i's$ in $T$.
\item A generalized Young tableau is called \emph{semi standard}, if its rows weakly increase and its columns strictly increase.
\item For $\lambda,\mu\vdash n$ the \emph{Kostka number} $K_{\lambda\mu}$ is defined as the number of semi standard Young tableaux of shape $\lambda$ and content $\mu$.
\end{enumerate}
\end{definition}
The significance of these definitions lies in the following
statement which originates from Young's work:
\begin{prop}\label{thm:youngrule}
For a partition
$\mu\vdash n$, the permutation module $M^\mu$ can be decomposed as
$$M^{\mu} \ = \ \bigoplus_{\lambda\unrhd \mu}K_{\lambda\mu}S^{\lambda}.$$
\end{prop}

\section{Symmetry-adapted relaxation}\label{se:symadapt}

As a first possibility to exploit symmetries in the framework of polynomial 
optimization, we consider the relaxation scheme introduced in 
Section~\ref{se:lasserre}.
Throughout the section, let $f,g_1, \ldots, g_m \in \R[X]$ 
and $K$ be the feasible set~\eqref{setk}.

The relaxation scheme yields a sequence of semidefinite programs.
While it is one possibility to exploit the symmetries on the level of the SDP,
the approach in this paper is to approach it on the level on top of this, i.e.,
on the level of the polynomials. To the best of our knowledge, in the framework of the relaxation scheme, no detailed
investigation has been made on the use of other polynomial bases (different from the standard
monomial basis). Here, the formulation of the relaxation scheme in the symmetry-adapted
basis will yield us symmetry-adapted versions of Putinar's Theorems 
(Theorems \ref{eq:SymMom} and \ref{thm:symPutinar})
and a symmetry-adapted relaxation scheme that converges
(Theorem \ref{thm:symscheme}).

In the following assume that $G$ is a finite group, although most of the 
results could be generalized to compact groups. We start by considering
linear group actions of $G$ on $\R^n$: For a set $S \subseteq \R^n$ and $\sigma \in G$,
let $\sigma(S) = \{x\in\R^n\, :\, \sigma^{-1}(x)\in S\}$.
By setting $p^{\sigma}(x):=p(\sigma^{-1}(x))$ for any polynomial $p \in \R[X]$
and $x \in \R^n$, $G$ induces a group action on $\R[X]$. If $p^{\sigma} = p$
for all $\sigma \in G$, then the polynomial $p$ is called $G$-\emph{invariant}.

Following Section~\ref{se:pre}, we now consider the probability measures
$\mathcal{P}(S)$ supported on a set $S$.
For $\mu \in \mathcal{P}(S)$,
$\sigma \in G$ and $\sigma^{-1}(S) \subseteq S$ define $\mu^{\sigma}$ by
$\mu^{\sigma}(B) = \mu(\sigma^{-1}(B))$ for any Borel set $B \subseteq S$.
A measure $\mu \in \mathcal{P}(S)$ is said to be {\it $G$-invariant} if $\mu = \mu^{\sigma}$
for all $\sigma \in G$, and the subset of all $G$-invariant probability measures on 
$S$ is denoted 
by $\mathcal{P}(S)^{G}$. For a comprehensive foundational treatment
of invariant measures we refer to \cite{cimpric}. 
Here, we mainly need the subsequent simple connection,
where for a set $S \subseteq \R^n$ we define
$S^{G} \ = \ \bigcap_{\sigma\in G}\sigma(S) \, .$
A set $S$ is called $G$-\emph{invariant} if 
$S=S^{G}$. Note, however that a $G$-invariant feasible set $K$ 
does not necessarily require that any of its defining polynomials
$g_j$ is $G$-invariant.
\begin{lemma}
Let the feasible set $K \subseteq \R^n$ be $G$-invariant. 
If $h\in  \R[X]$ is
$G$-invariant then
$$\inf_{x\in K} h(x) \ = \ \inf_{\mu\in\mathcal{P}^G(K)}\int_{K} h \, d\mu \, .$$
\end{lemma}

\begin{proof}
We have 
\begin{eqnarray*}
h^* \ := \ \inf_{x\in K}h(x) \ = \ \inf_{\mu\in\mathcal{P}(K)}\int h \, d\mu&\leq&
\inf_{\mu\in\mathcal{P}(K^G)}\int h \, d\mu \, , \quad\mbox{as $\mathcal{P}(K)^G \subseteq\mathcal{P}(K)$} \, .
\end{eqnarray*}
Let $(x_k)$ be a minimizing sequence in $K$ such that $h(x_k) \rightarrow
h^*$ as $k \rightarrow\infty$.
To each $x_{k}$ we can define a Dirac measure $\mu_{k}$ supported in $x_{k}$. Now this gives a converging  sequence $(\int h(x)d\mu_k)$.  
The measure $\mu_{k}^{*}:=\frac{1}{|G|}\sum_{\sigma\in G} \mu_k^{\sigma}$
is contained in $\mathcal{P}(K)^G$ for every $k$. Since $h$ is $G$-invariant, $\int h(x)d\mu_k^{*}=h(x_k)$ which in turn implies
$\int h(x)d\mu_k^{*}\to h^*\leq \inf_{\mu\in\mathcal{P}(K)^G}\int fd\mu$,
and so $h^*=\inf_{\mu\in\mathcal{P}(K)^G}\int fd\mu$.
\end{proof}

So in order to find the infimum of a $G$-invariant function $f$ on 
a $G$-invariant set $K$ we only have to consider the invariant measures supported on $K$. 
Hence to make a relaxation scheme for this setting similar to the one presented in Section~\ref{se:lasserre}, it suffices to consider \emph{$G$-linear}
maps $L^G:\, \R[X]\to \R$, i.e., linear maps with $L^G(f)=L^G(f^{\sigma})$ for all 
$f \in \R[X]$ and $\sigma\in G$.
In analogy to Putinar's Theorem~\ref{th:putinar} we can also characterize them in terms of bilinear forms.

\begin{thm}\label{eq:SymMom}
Let $g_1,\ldots,g_m\in\R[X]$ be $G$-invariant, and assume that the
feasible set $K$ satisfies Assumption~\ref{putinar-assumption}.
Setting $g_0:=1$, a $G$-linear map $L^{G}: \R[X] \to \R$ is the
integration with respect to a $G$-invariant measure on $K$ if and only if the
bilinear forms
\begin{equation}
\label{eq:bilinpsd}
\begin{array}{rcl}
\mathcal{L}_{g_{j}}^{G} \: : \: \R[X]\times \R[X]&\to&\R\\
(p,q)&\mapsto&L^{G} \Big( \frac{1}{|G|}\sum\limits_{\sigma\in G}(p\cdot q)^{\sigma} \cdot g_j \Big)
\end{array}
\end{equation}
are psd for all $0 \le j \le m$.
\end{thm}

Note that the definition of $\mathcal{L}^G_{g_j}$ only uses the values of 
$L^G$ on
the invariant ring $\R[X]^G \subseteq \R[X]$. 
Indeed, any $G$-linear map $L^G : \R[X] \to \R$
is already completely determined by its values on $\R[X]^G$, 
since for any $p \in \R[X]$
we have $L(p) = \frac{1}{|G|} \sum_{\sigma \in G} L^G(p^{\sigma}) =
L^G(\frac{1}{|G|} \sum_{\sigma \in G} p^{\sigma})$.

\begin{proof}
{\it Only if part.} Let $L_\mu\,: \R[X] \rightarrow \R$ be the linear functional associated with a
$G$-invariant measure $\mu$, $L_\mu(f) :=\int_{K} f \, d\mu$. Then
\begin{eqnarray*}
L_\mu\left(\frac{1}{|G|}\sum_{\sigma\in G}(f^2)^{\sigma}\cdot g_j\right)&=&L_\mu\left(\frac{1}{|G|}\sum_{\sigma\in G}(f^2\cdot g_j)^\sigma\right)=\frac{1}{|G|}\sum_{\sigma\in G}\int_K(f^2\cdot g_j)^\sigma\, d\mu\\
&=&\frac{1}{|G|}\sum_{\sigma\in G}\int_K(f^2\cdot g_j) \, d\mu^{\sigma}\\
&=&\frac{1}{|G|}\sum_{\sigma\in G}\int_K(f^2\cdot g_j) \, d\mu \quad\mbox{as $\mu^{\sigma}=\mu$}\\
&\geq&0 \quad\mbox{as $g_j\geq 0$ on $K$.}
\end{eqnarray*}

{\it If part.} Conversely, let $L^G$ be as in the statement of the theorem. Then for all $h\in\R[X]$,
\begin{eqnarray*}
0 \ \leq \ \mathcal{L}_{g_j}^G(h^2)&=&L^G\left(\frac{1}{|G|}\sum_{\sigma\in G}(h^2)^{\sigma}g_j\right)=L^G\left(\frac{1}{|G|}\sum_{\sigma\in G}(h^2 g_j)^\sigma\right)\\
&=&\frac{1}{|G|}\sum_{\sigma\in G}L^{G}((h^2 g_j)^\sigma)\\
&=&\frac{1}{|G|}\sum_{\sigma \in G}L^G(h^2 g_j)\quad\mbox{as $L^G$ is $G$-invariant.}
\end{eqnarray*}
Since this holds for all $h\in  \R[X]$ and all $j = 1, \ldots, m$, by Putinar's Theorem $L^G$ is the
integration with respect to some measure $\mu$ on $K$. In addition, for every $h\in  \R[X]$
and every $\sigma\in G$,
$$\int_K h\, d\mu \ = \ L^G (h) \ = \ L^G(h^\sigma) \ =\ \int_K h^\sigma \, d\mu \ = \ 
\int _K h \, d\mu^\sigma \, ,$$
and so as $K$ is compact, $\mu = \mu^{\sigma}$ for all $\sigma\in G$, i.e., $\mu$ is $G$-invariant.
\end{proof}

So the $G$-invariant optimization problem~\eqref{eq:opt1} can be rephrased as
\begin{equation}\label{eq:sym}
p^{*} \ = \ \inf\left\{L^G(p)\: : \: L^G \text{ a $G$-linear map } \R[X]^G\rightarrow \R \, , \; L^G(1)=1 \text{ and each } \mathcal{L}_{g_j}^{G} \mbox{is psd}\right\}.
\end{equation}

The reformulation gives already a first computational advantage compared to the usual approach.
By the definition of $\mathcal{L}_{g_{j}}^{G}$ in ~\eqref{eq:bilinpsd} we can phrase
the moment matrix in terms of variables which are merely indexed by a basis $\mathcal{B}$ of the
invariant ring $\R[X]^G$. This reduction on the number of moment variables is
illustrated in the following example.

\begin{ex}\label{ex:cyclic}
Let $C_4$ denote the cyclic group of order four that operates on $\R^4$ by cyclically permuting the coordinates. 
The space of $C_4$-invariant polynomials of degree at most  2 is spanned by $b_0:=1$,
$b_1:=\frac{1}{4}(x_1+x_2+x_3+x_4)$, $b_2:=\frac{1}{4}(x_1^2+x_2^2+x_3^2+x_4^2)$, $b_3:=\frac{1}{4}(x_1x_2+x_2x_3+x_3x_4+x_4x_1)$,
$b_4:=\frac{1}{2}(x_1x_3+x_2x_4)$. By identifying 
$y_{i} = L^G(b_i)$ with moment variables $y_0, \ldots, y_4$,
the (truncated) psd condition~\eqref{eq:bilinpsd} for polynomials 
of degree at most~2 can be stated as
a matrix psd condition in only four variables $y_{1}, \ldots, y_{4}$,
$$\ \left( \begin {array}{ccccc} 1 &y_{1}&y_{1}&y_{1}&y_{1}\\
                  \noalign{\medskip}y_{1}&y_{2}&y_{3}&y_{4}&y_{3}\\
                 \noalign{\medskip}y_{1}&y_{3}&y_{2}&y_{3}&y_{4}\\
                 \noalign{\medskip}y_{1}&y_{4}&y_{3}&y_{2}&y_{3}\\
                 \noalign{\medskip}y_{1}&y_{3}&y_{4}&y_{3}&y_{2}\end {array}
\right) \ \succeq \ 0 \, .$$

\end{ex}

In addition to this reduction of the number of variables, the structure of the moment matrix approach can be simplified based on representation theory. 
In order to apply the methods from Section~\ref{se:representation},
observe that for any $k \ge 0$,
the subset (of $\R[X]$) of polynomials of total degree at most $k$
is finite-dimensional and thus can be viewed as a real $G$-module. 
As a consequence of this exhaustion process for $\R[X]$, there exists a 
complex decomposition of the form
\begin{equation}
\label{eq:complexdecomp}
  \R[X] \otimes \C \ = \ \bigoplus_{i=1}^h V_i \ = \ 
    \bigoplus_{i=1}^h \bigoplus_{j \in J_i} W_{ij}
\end{equation}
with complex irreducible components $W_{ij}$,
and corresponding real decomposition of the form~\eqref{eq:realdecom},
\begin{align}\label{eq:decomp}
\R[X] \ = \ (V_1+V_2)\oplus \frac{1}{\mathrm{i}}(V_1-V_2)\oplus \cdots\oplus 
(V_{2l-1} +V_{2l}) \oplus
\frac{1}{\mathrm{i}}(V_{2l-1} -V_{2l})\oplus V_{2l+1}\oplus\ldots\oplus V_{h} \, ,
\end{align}
where $\mathrm{i}$ is the imaginary unit.
Note that the index sets $J_1, \ldots, J_h$ may be infinite now.
The component with respect to the trivial irreducible
representation is the invariant ring $\R[X]^G$, and the elements of the
other isotypic components are called \emph{semi-invariants}. 

Fix an $i \in \{1, \ldots, h\}$. 
Since the $W_{ij}$ are pairwise isomorphic,
we can choose $G$-isomorphisms $\phi^i_j : W_{i1} \to W_{ij}$, 
$j \in J_i$.
Further we can assume that $W_{i1}$ is generated by a basis
$\{ s^{i}_{1,u} \, : \, 1 \le u \le \dim W_{ij}\} \subseteq \C[X]$, 
such that any
$s^{i}_{1,u}$ is in the $G$-orbit of $s^{i}_{1,1}$.
The basis vector $s^{i}_{1,1}$ transfers to
$W_{ij}$ by the $G$-isomorphism $\phi^i_j$ via
$s_{j,1}^i := \phi^i_j(s_{1,1}^i)$, and thus by this isomorphism $\phi^i_j$
the whole basis $\{ s^{i}_{1,u} \, : \, 1 \le u \le \dim W_{i1} \}$
of $W_{i1}$ transfers to a whole basis 
$\{s^{i}_{j,u} \, : \, 1 \le u \le \dim W_{ij} \}$ of $W_{ij}$.
Set $\mathcal{S}^i = \{s^i_{j,1} \, : \, j \in J_i \}$.

Using the bookkeeping techniques from Section~\ref{se:representation}
(describing the transition from~\eqref{eq:complexdecomp} to~\eqref{eq:decomp}),
the set $\mathcal{S}^i \subseteq \C[X]$ 
can be transformed into a subset of $\R[X]$ by distinguishing
the types I, II, and III. Therefore without loss of generality we can assume 
$\mathcal{S}^i$ to be real.

\begin{thm} \label{th:putinarsym1}
Let $g_1,\ldots,g_m \in \R[X]$ be $G$-invariant, assume that
the feasible set $K$ satisfies Assumption~\ref{putinar-assumption}, 
and set $g_0 :=1$.
A $G$-linear map $L^{G}:\R[X]\to\R$ is the integration with respect to a $G$-invariant 
measure $\mu$ on $K$ if and only if for all $i \in \{1, \ldots, h\}$ and all
$j \in \{1, \ldots, m\}$ the bilinear map
$\mathcal{L}^{G}_{g_j}$ from~\eqref{eq:bilinpsd} restricted to
$\mathcal{S}^{i}$ $(\subseteq \R[X])$ is positive semidefinite.
\end{thm}

\begin{proof}
By Theorem~\ref{eq:SymMom}, $L^G$ is the integration w.r.t.\ a measure $\mu$ supported on $K$ if and only
if all the bilinear forms $\mathcal{L}^{G}_{g_j}$ from~\eqref{eq:bilinpsd} are
psd. Clearly, the latter condition implies that all the restrictions to the
$\mathcal{S}^i$ are psd. 

Conversely, to keep indices simple, fix one of the polynomials
$g_0, \ldots, g_m$ and call it shortly $g$.
By assumption, the restriction of $\mathcal{L}^G_{g}$ to any 
$\mathcal{S}^i$ is psd, $1 \le i\le h$. 
We show that this already implies that $\mathcal{L}^{G}_{g}$ is psd.
To this end,  we first show that for distinct real
isotypic components $V_l$ and $V_k$ we have $L^G(p_l \cdot p_k \cdot g)=0$ 
for all $p_l\in V_l$ and $p_k\in V_k$. 
Indeed, each $p_l$ defines a linear map
\[ 
  \chi_{p_l}: V_k\rightarrow \R \, , \quad 
  q \mapsto \mathcal{L}^G_{g} (p_l,q) \, .
\]
Hence, the application $p_l\mapsto \chi_{p_l}$ gives rise to a $G$-homomorphism from 
$V_l$ to the dual space $V_k^{*}$ (and thus to a $G$-homomorphism from $V_l$
to $V_k$.)
But by Schur's Lemma \ref{le:Schur} this has to be the zero map and thus 
$\mathcal{L}^G_{g}(p_l,q)=0$\ for all $q\in V_k$. 
  
Consequently  it suffices to  show that for every $i \in \{1, \ldots, h\}$ the positive semidefiniteness on $\mathcal{S}^i$ implies already positive semidefiniteness on $V_i$.
We first assume that $V_i$ comes from a complex irreducible representation of
type {\bf I}.
Consider a pair $W_{ij}, W_{ik}$ in the 
decomposition~\eqref{eq:complexdecomp}, 
where we allow $j=k$. Similar to the above reasoning,
Schur's Lemma can be applied by considering a linear map 
$\psi_{{j,k}}: W_{ij}\rightarrow W_{ik}$ defined by its images
on the basis vectors,
$$\psi_{j,k}(s^{i}_{j,u})
 \ := \ \sum_{v}  \mathcal{L}^G_{g}(s^{i}_{j,u}, s^{i}_{k,v}) \cdot s^i_{k,v} \, , \quad 1 \le u \le \dim W_{ij} \, .$$
Since by Schur's Lemma a $G$-isomorphism from
$W_{ij}$ to $W_{ik}$ is unique up to a scalar multiplication we can
write this map as a composition
$\psi_{j,k}=c_{jk}\cdot \phi^i_k \circ (\phi^i_j)^{-1}$
where the constant $c_{jk}$ can be chosen real. This implies
$$\mathcal{L}^G_{g}(s^{i}_{j,u},s^{i}_{k,v}) \ = \ \delta_{uv}c_{jk} \, ,$$
where $\delta_{uv}$ denotes Kronecker's delta function. Since any  
$f\in V_i$ can
be written in the form $f=\sum_{j}\sum_{u}\alpha_{j,u}s^{i}_{j,u}$ 
with $\alpha_{j,u} \in \R$, the $G$-invariance of $\mathcal{L}_g^G$
gives
$$\mathcal{L}_{g}^G(f,f) \ = \ \mathcal{L}^G_{g} \big( \sum_{j}\sum_{u}\alpha_{j,u}s_{j,u}^{i},\sum_{j}\sum_{u}\alpha_{j,u}s^{i}_{j,u} \big) 
\ = \ \sum_{u}\mathcal{L}^G_{g} \big( \sum_{j}\alpha_{j,u}s^{i}_{j,1},\sum_{j}\alpha_{j,u}s_{j,1}^i \big) \, ,$$
which is nonnegative since $\mathcal{L}^G_{g}$ is psd on $\mathcal{S}^i$. 
For the other types ({\bf II}  and {\bf III}) the result is implied in the same
 way by taking additionally into 
 account the bookkeeping techniques explained
 in Section~\ref{se:representation} (see also Example \ref{ex:complex} below). 
Therefore the forms $\mathcal{L}^G_{g}$ are psd on $\R[X]$
if and only if the restrictions to each $\mathcal{S}^i$ are psd and the statement follows with Theorem~\ref{eq:SymMom}.
\end{proof}
We also record the following symmetric version of  Putinar's Positivstellensatz,
which follows from dualizing Theorem~\ref{th:putinarsym1}.

\begin{thm}\label{thm:symPutinar}
Let $f,g_1,\ldots,g_m\in\R[X]$ be $G$-invariant, and let the feasible
set $K$ satisfy Assumption~\ref{putinar-assumption}. If $f$ is strictly positive  on $K$, then
$f$ can be written in the form
$$f \ = \ \rho^G \Big( \sum_{i=1}^h q_0^i + 
  \sum_{j=1}^m g_j \sum_{i=1}^h q_j^i \Big)
$$
where $q_j^i$ is a sum of squares of polynomials in $\mathcal{S}^i$,
$1 \le i \le h$,
$1 \le j \le m$, and 
$\rho^G(p) = \frac{1}{|G|} \sum_{\sigma \in G} p^{\sigma}$.
\end{thm}

Putting all this together we obtain the following symmetry-adapted
relaxation scheme. 
For every $k\in \N$ let $\mathcal{B}_{2k}$ be a basis of the set
$\R[X]^G_{\le 2k}$ of $G$-invariant polynomials of degree at most $k$,
with $1 \in \mathcal{B}_{2k}$.
Any $G$-invariant map  $L^G:\;\R[X]^G_{\le 2k}\to  \R$ can be
identified with a finite 
sequence $y = (y_b)_{b \in \mathcal{B}_{2k}}$ 
by setting $y_b:=L^G(b)$ for $b\in\mathcal{B}_{2k}$. 
With regard to the real 
decomposition of $\R[X]_{\le k}$
coming from the truncation of~\eqref{eq:decomp},
construct the sets $\mathcal{S}^{i}_k
=\{s^{i}_1,s^{i}_2,\ldots,s^{i}_{\eta_i}\}\subseteq\mathcal{S}^i$ of the
basis elements of $\mathcal{S}^i$  of degree at most $k$.
Then we define the \emph{symmetry-adapted moment matrix} $M^{G}_k(y)$ by
\begin{equation}
\label{eq:symmmomentmatrix}
M_k^G(y) \ :=\ \bigoplus_{i=1}^hM_{k,i}^G(y), \text{ where } M^G_{k,i}(y)_{v,w} \ := \ \mathcal{L}_{g_0}^G(s^{i}_v,s^{i}_w) \ = \ \mathcal{L}_{1}^G(s^{i}_v,s^{i}_w) \, .
\end{equation}
The entries of $M^G_k(y)$ are linear combinations of the elements of $y$ 
and hence indexed by elements in $\mathcal{B}_{2k}$.
Defining the \emph{symmetry-adapted localizing matrices} in a similar manner,
we obtain the symmetry-adapted relaxation for 
$k \ge k_0 := \max \{ \lceil \deg f/2 \rceil, 
  \lceil \deg g_1/2 \rceil,$ \ldots,$\lceil \deg g_m / 2 \rceil \}$,
\begin{equation}\label{eq:lasserresym}
 Q^G_k:\quad\begin{array}{rcl}
  \inf_y\,L^G(f)\\
 M_k^G(y) & \succeq & 0 \, , \\
 M_{k - \lceil \deg g_j / 2 \rceil}^G(g_j \, y) & \succeq & 0 \, , \quad
1 \le j \le m \, , \\
 y_1 & = & 1 \, ,
 \end{array}
\end{equation}
with optimal value denoted by $\inf Q^G_k$ (and $\min Q^G_k$ if the infimum
is attained).
\begin{remark}\label{re:size}\emph{Computational aspects:}
The symmetry-adapted setting defined above can give a significant reduction 
compared to the original relaxation scheme. 
Indeed the number of variables involved equals the size of $\mathcal{B}_{2k}$.
Furthermore the symmetry-adapted moment matrix is block diagonal with blocks
of sizes $\eta_1, \ldots, \eta_h$. 

If the irreducible representations of a given group are known 
then isotypic decompositions and therefore the basis polynomials $s_j^i$ can 
be algorithmically computed using projections (see, \cite[Prop.~8]{serre-b77}).

With regard to determining the irreducible components, there
are theoretical methods whose computational complexity is bounded by a
polynomial in $|G|$ (see \cite{babai-ronyai-90}). 
For practical purposes, the size $|G|$ of the group might not
be the appropriate measure of complexity, and indeed, there are
practical methods (see \cite{dabbaghian-2005,dixon-93})
which are implemented in the group-theoretic software GAP \cite{GAP}.

Note that all these pre-computations have to be done only once for a 
specific group and relaxation order.
\end{remark}

In this setting, Proposition
\ref{pr:lasserreconv} can be reformulated as follows.
\begin{thm}\label{thm:symscheme}
Let $f, g_1, \ldots, g_m \ \R[X]$ be $G$-invariant, 
let Assumption \ref{putinar-assumption} holds for the feasible set $K$,
and let
$(Q_k^G)_{k\geq k_0}$ be the hierarchy~\eqref{eq:lasserresym} of
symmetry-adapted SDP-relax\-ations.
Then $(\inf Q^G_k)_{k\geq k_0}$ is a monotonically non-decreasing sequence
that converges to $f^*$.

\end{thm}
\begin{proof}
As $\mathcal{P}(K)^{G}\subseteq\mathcal{P}(K)$ one has $\inf Q_k^G\geq \inf Q_k$ for all $k\geq k_0$.
In addition, for any measure $\mu$ on $K$ we let $\mu^{\#} \ = \ \frac{1}{\left| G \right|} \sum_{\sigma \in G} \mu^{\sigma}$. As $K$ is $G$-invariant,
$\mu^{\#}$ is also supported on $K$.
This proves that $\inf Q_k^G\leq f^*$ for all $k\geq k_0$, and so,
$\inf Q_k\leq \inf Q^G_k\leq f^*$ for all $k\geq k_0$. Combining
the latter with Proposition \ref{pr:lasserreconv} yields the desired
result.
\end{proof}
\begin{remark} If -- more general than the setting considered in this article --
not all $g_j$ are $G$-invariant but the set $K$ is $G$-invariant or even if just the set of optimal values is invariant, it is still possible to only look at invariant moments. However the above block structure will 
in general only apply for the moment matrix and the localizing matrices for the $G$-invariant polynomials. Note, however, that the variables in the localizing matrices still correspond to a basis for the space of $G$-invariants.
\end{remark}

The general setting presented in this section leads to the question of handling the
hierarchy of symmetry-adapted bases for the hierarchy of vector spaces of polynomials.
Before studying in detail the symmetric group in the next section,
as a warm up, it is worth to revisit Example \ref{ex:cyclic}.

\begin{example}\label{ex:complex}
We continue Example~\ref{ex:cyclic}.
Since the cyclic group $C_4$ is abelian, 
all the irreducible representations are one-dimensional and correspond to the 
fourth roots of unity. Consider the symmetry-adapted relaxation for $k=1$.
Using the methods from Remark~\ref{re:size} we find
$\R[X]_1\otimes\C=\bigoplus_{l=1}^4 V_{l}$, where
$V_{l}= \myspan \{ \sum_{j=1}^{4}\omega_l^j X_j\}$ and $\omega_1, \ldots, \omega_4$ denote the fourth roots of unity.
Thus, with regard to the symmetry-adapted moment matrix~\eqref{eq:symmmomentmatrix}, we obtain
$\mathcal{S}^{2l-1}_1 = \myspan \{\sum_{j=1}^{4}(\omega_l^j +\overline{\omega_l}^j) X_j\}$ and
$\mathcal{S}^{2l}_1 = \myspan \{\frac{1}{\mathrm{i}}\sum_{j=1}^{4}(\omega_l^j -\overline{\omega_l}^j) X_j\}$, where $1 \le l \le 2$.
With the notation of Example \ref{ex:cyclic} the truncated 
symmetry-adapted moment matrix is
$$\left( \begin {array}{ccccc} 1&2\,y_{1}&0&0&0\\\noalign{\medskip}2\,y_{1}&y_{2}+2\,y_{3}
+y_{4}&0&0&0\\\noalign{\medskip}0&0&y_{2}-y_{4}&0&0\\\noalign{\medskip}0&0&0&y_{2}-2\,
y_{3}+y_{4}&0\\\noalign{\medskip}0&0&0&0&y_{2}-y_{4}\end {array} \right) \, . 
$$
We get 4 diagonal blocks (of which 3 are elementary) and so we end up with a $2\times 2$ semidefiniteness
constraint instead of a $5\times 5$ one if symmetry is not exploited. 

\end{example}

\section{Optimizing with symmetric Polynomials}\label{se:sympol}

In this section, we provide several techniques to exploit symmetries for the symmetric group $\mathcal{S}_n$.
While the representation theory of the symmetric group is a classical topic (as reviewed in
Subsection~\ref{se:prelimsymmgroup}), it yields some interesting (even somewhat surprising) results
in our setting.

First, in Section~\ref{se:momentmatrixsymmetric} we discuss the 
symmetry-adapted relaxation for the symmetric
group. By realizing the irreducible components in a suitable basis of polynomials
(generalized Specht polynomials as defined below),
the moment matrix can be characterized rather explicitly (Theorem \ref{thm:symmom}).
As corollaries, we derive some concrete representation theorems for symmetric polynomials in
Section~\ref{se:representationtheorems}.

\subsection{Moment matrices for the symmetric group\label{se:momentmatrixsymmetric}}
Recall from the preliminaries that the irreducible representations of $\mathcal{S}_n$ are in natural bijection with the partitions of $n$. In order to construct a suitable generalized moment matrix we will need a graded decomposition of the vector space $\R[X]$ into $\mathcal{S}_n$-irreducible components. A classical construction of Specht gives a realization of the Specht modules as polynomials (see \cite{specht-1933}):

For $\lambda\vdash n$ let $t$ be a $\lambda$-tableau. To $t$
we associate the monomial $X^{t}:=\prod_{i=1}^{n}X_i^{l(i)-1}$, 
where $l(i)$ is the index of the row of $t$ containing~$i$.
Denote by $\mathcal{C}_1,\ldots,\mathcal{C}_{\nu}$ the columns of $t$
and by $\mathcal{C}_j(i)$ the element in the $i$-th row of the 
column $\mathcal{C}_j$. Then we associate to each column $\mathcal{C}_j$
a Vandermonde determinant 
$$\Van_{\mathcal{C}_{j}} \ := \ \det
\left(
\begin{array}{ccc}
X_{ \mathcal{C}_j(1)}^0& \ldots  &X_{\mathcal{C}_j(r_j)}^0   \\
 \vdots&  \ddots &\vdots   \\
X_{ \mathcal{C}_j(1)}^{r_j-1}& \ldots  &X_{\mathcal{C}_j(r_j)}^{r_j-1}
\end{array}
\right) \ = \ \prod_{1 \le i<l \le  r_j}(X_{\mathcal{C}_j(l)}-X_{\mathcal{C}_j(i)}) \,$$
where $r_j$ denotes the number of 
rows of $\mathcal{C}_j$.

The \emph{Specht polynomial} $s_{t}$ associated to $t$ is defined as
\[
  s_{t} \ := \ \prod_{j=1}^{\nu} \Van_{\mathcal{C}_{j}}
  \ = \
 \sum_{\sigma\in \CStab_{t}}\sgn(\sigma)\sigma(X^{t}) \, ,
\]
where $\CStab_{t}$ is the column stabilizer of $t$
as introduced in Section~\ref{se:prelimsymmgroup}.
The polynomials
$\{ s_{t} \, : \, t 
  \text{ standard Young tableau to $\lambda$} \}$
are called the \emph{Specht polynomials} associated to $\lambda$.

Note that for any $\lambda$-tabloid $\{t\}$ the monomial 
$X^{t}$ is well defined,
and the mapping $\{t\} \mapsto X^{t}$ is an 
$\mathcal{S}_n$-invariant mapping. Thus
$\mathcal{S}_n$ operates on $s_{t}$ in the same way as on the polytabloid $e_{t}$. This observation implies (see \cite{specht-1933}):

\begin{lemma} \label{pr:spechtpoly}
For any partition $\lambda \vdash n$,
the Specht polynomials associated to $\lambda$
span an 
$\mathcal{S}_n$-submodule of $\R[X]$ which is isomorphic to the Specht module $S^{\lambda}$.
\end{lemma}

While Lemma~\ref{pr:spechtpoly}
already gives a realization of the Specht modules in terms of
polynomials, for the symmetry-adapted moment matrix we need
to generalize this construction to realize these modules in terms of polynomials
with prescribed exponent vectors.
In the following, let $n\in \N$ and $\beta:=(\beta_1,\ldots,\beta_n)$ be an $n$-tuple
of non-negative integers, and set
$\R\{X^\beta\} := \myspan \{ \sigma(X^{\beta}) \, : \, \sigma \in \mathcal{S}_n\}$.
By construction, $\R\{X^\beta\}$ is closed under
the action of $\mathcal{S}_n$ and therefore has the structure of an $\mathcal{S}_n$-module.

Denote by $\wt(\beta)=\sum_{i=1}^{n}\beta_i$ the \emph{weight} of $\beta$.
Let $b_1, \ldots, b_{\ell}$ be
the distinct components of $\beta$ (called the \emph{parts} of $\beta$),
ordered (decreasingly) according to
the multiplicity of the occurrence in $\beta$. Further let
$I_j = \{ i \in \{1, \ldots, n\} \, : \, \beta_i =b_j \}$, $1 \le j \le \ell$,
and note that the sets 
$I_1, \ldots, I_\ell$
define a partition of $\{1, \ldots, n\}$.
Setting $\mu_j:=|I_j|$, the vector $\mu=(\mu_1,\ldots,\mu_\ell)$ consists of monotonically
decreasing components and thus defines a partition of $n$. We call $\mu\vdash n$ the \emph{shape} of $\beta$.

\begin{lemma}\label{le:M}
For $\beta \in \N_0^n$, the
$\mathcal{S}_n$-module $\R\{X^\beta\}$ is isomorphic to the permutation module $M^{\mu}$, where $\mu$ is the shape of $\beta$.
\end{lemma}

\begin{proof}
We construct an isomorphism from $\R\{X^\beta\}$ to $M^{\mu}$ by defining
it on the basis elements.

First observe that $\R\{X^{\beta}\} = \myspan \{ X^{\gamma} \, : \, \gamma \text{ permutation of } \beta \}$. For a fixed 
permutation $\gamma$ of $\beta$, consider the 
set partition $I_1,\ldots,I_{\ell}$ associated to $\gamma$,
and map $X^{\gamma}$ to the $\mu$-tabloid with rows $I_1, \ldots, I_{\ell}$.
Since this mapping commutes with the action of $\mathcal{S}_n$,
and since by Definition~\ref{def:Permutationmodule} the $\mathcal{S}_n$-module 
$M^{\mu}$ is spanned by 
the set of all $\mu$-tabloids, the statement follows.
\end{proof}

Now let $\lambda\vdash n$ be another partition of $n$. In order to construct the realizations of the Specht module $S^{\mu}$ as submodules of $\R\{X^\beta\}$, we look at pairs $(t, T)$, where $t$ is a fixed $\lambda$-tableau and $T$ is a generalized Young tableau with shape $\lambda$ and content $\mu$.
For each  pair we construct a monomial
$X^{(t,T)} \in \R\{X^{\beta}\}$ from its parts
$b_1,\ldots, b_\ell$ in the following way.
As before, let $\CC_1,\ldots, \CC_\nu$ be the columns of  $t$,
and denote by $T(i,j)$ the element in the $i$-th row and $j$-th column 
of $T$.
Then define 
$$X^{(t,T)} \ := \ \prod_{(i,j)}X_{\CC_j(i)}^{b_{T(i,j)}} \, ,$$
and associate to each column $\CC_j$ a polynomial
\begin{equation}
\label{eq:vandermonde2}
\Van_{\mathcal{C}_{j},T} \ := \ \det
\left(
\begin{array}{ccc}
X_{ \mathcal{C}_j(1)}^{b_{T(1,j)}}& \ldots  &X_{\mathcal{C}_j(k)}^{b_{T(1,j)}}   \\
\vdots&  \ddots &\vdots   \\
X_{ \mathcal{C}_j(1)}^{b_{T(k,j)}}& \ldots  &X_{\mathcal{C}_j(k)}^{b_{T(k,j)}}
\end{array}
\right) \, .
\end{equation}
As in Specht's construction we form the product polynomial 
$s_{(t,T)} \ := \ \prod_{j=1}^{\nu}\Van_{\mathcal{C}_{j},T}$,
and set (by summation over the row equivalence class $\{T\}$ of $T$)
$S_{(t,T)} \ :=\ \sum_{S\in\{T\}}s_{(t,S)} \, .$

\begin{lemma} \label{le:Spechtoccurrence}
Let $\lambda \vdash n$, $\beta \in \N_0^n$, and
$\mu$ be the shape of $\beta$. Further let
$t$ be a $\lambda$-tableau and
$T$ be a generalized Young tableau with shape $\lambda$ and content $\mu$.
The $\mathcal{S}_n$-submodule $\R\{S_{(t,T)}\}$ of $\R\{X^{\beta}\}$
generated by the generalized Specht polynomial $S_{(t,T)}$ 
is isomorphic to the Specht module $S^{\lambda}$.
\end{lemma}

\begin{proof}
By Lemma~\ref{le:M} we can follow Young's decomposition of $M^\mu$. Therefore we associate to every $T$ with shape $\lambda$ and content $\mu$ an $\mathcal{S}_n$-homomorphism $\Theta_{T,t} \, : \,
M^\lambda \to \R\{X^{\beta}\}$, which maps the given $\lambda$-tableau
$\{t\}$ to $\sum_{S\in\{T\}}X^{(t,S)}$
and which naturally extends by the cyclic structure of $M^\lambda$.

Now it suffices to show that the image of 
$S^{\lambda} \subseteq M^{\lambda}$
under $\Theta_{T,t}$
coincides with $\R\{S_{(t,T)}\}$. The image of the
polytabloid $e_{t}$ from~\eqref{eq:polytabloid} is
$$\Theta_{T,t}(e_{t}) \ = \ 
\Theta_{T,t} \Big( \sum_{\sigma\in\CStab_{t}}\sgn(\sigma)\sigma\{t\} \Big) = \ \sum_{\sigma\in\CStab_{t}} \Theta_{T,t} (\sgn(\sigma)\sigma\{t\}) \, . $$ 
Since by the Leibniz expansion of~\eqref{eq:vandermonde2}
we have $s_{(t,T)}=\sum_{\sigma\in \CStab_{t}}\sgn(\sigma)\sigma(X^{(t,T)})$, it follows that 
$\Theta_{T,t}(e_{t}) = S_{(t,T)}$,
which proves the claim.
\end{proof}

\begin{remark}
Note the following connection of the generalized Specht polynomials to the classical Schur polynomials.
For a non-negative vector $\lambda = (\lambda_1, \ldots, \lambda_l)$ the
generalized Vandermonde determinant
\begin{equation}
\label{eq:alternant}
a_{\lambda} \ := \ \det \left( (X_i^{\lambda_j})_{1 \le i, j \le l} \right)
\end{equation}
(as a polynomial in $X_1, \ldots, X_l)$ is called the \emph{alternant} of $\lambda$. Moreover, for a partition $\lambda$
of length $l$ the \emph{Schur function} $s_{\lambda}$ is defined by
$s_{\lambda} \ = \ a_{\lambda+\delta} /a_{\delta}$,
where $\delta := (l-1, l-2, \ldots, 1,0) \in \Z^l$. It is well known that $s_{\lambda}$ is a symmetric polynomial
in $X_1, \ldots, X_l$ (also called a \emph{Schur polynomial}). Hence, the alternant~\eqref{eq:alternant} can be written as
\begin{equation} \label{eq:amu}
a_{\lambda} \ = \ s_{\lambda-\delta} \cdot a_{\delta} \, .
\end{equation}

Now the polynomials $\Van_{\mathcal{C}_{j},T}$ defined above can be seen as the
alternant associated to the numbers $(b_{T(1,j)},\ldots,b_{T(k,j)})$  and thus by~\eqref{eq:amu} as the
product of a Schur polynomial with a classical Vandermonde determinant.
\end{remark}

Let $\mathcal{T}^{0}_{\lambda,\mu}$ denote the set of semi standard generalized Young tableaux of shape $\lambda$ and content $\mu$. To conclude we can summarize the above considerations.

\begin{thm}
\label{th:maindecomp}
Let $\beta \in \N_0^n$ of weight $d$ and shape $\mu = \mu(\beta)$.
Then we have
$$\R\{X^{\beta}\} \ = \ \bigoplus_{\lambda\unrhd \mu}\bigoplus_{T\in\mathcal{T}^{0}_{\lambda,\mu}}\R\{S_{(t_\lambda,T)}\},$$
where $t_\lambda$ denotes the unique $\lambda$-tableau with increasing rows and columns.
The multiplicity of the Specht modules $S^{\lambda}$ in this $\mathcal{S}_n$-module is equal to the Kostka number $K_{\lambda\mu}$.
\end{thm}

\begin{proof}
By Lemma~\ref{le:M} the $\mathcal{S}_n$-module 
$\R\{X^\beta\}$ is isomorphic to $M^{\mu}$.
By Proposition~\ref{thm:youngrule},
the multiplicity of $S^{\lambda}$ in $M^{\mu}$ is $K_{\lambda\mu}$,
corresponding to the decomposition 
$M^{\mu} = \bigoplus_{\lambda\unrhd \mu}$
  $\bigoplus_{T\in\mathcal{T}_{\lambda,\mu}^{0}}
  S^{\lambda}
$
in terms of semi standard Young tableaux. 
Now for $T\in \mathcal{T}^{0}_{\lambda,\mu}$ the $\mathcal{S}_n$-homomorphism
$\Theta_{T,t}$ constructed in the proof of Lemma~\ref{le:Spechtoccurrence}
maps $e_{t_{\lambda}}$ to $S_{(t_{\lambda},T)}$ and thus turns this decomposition
into the decomposition of $\R\{X^{\beta}\}$.
\end{proof}

Based on these results, we can construct the
symmetry-adapted moment matrix of order~$k$.
By~\eqref{eq:symmmomentmatrix} the blocks are labeled by partitions of $n$. In order to construct the block for a fixed $\lambda$ consider the various $\beta=(\beta_{1}, \ldots ,\beta_{n})$ with $\wt(\beta)=k$ and shape $\lambda$. For a given $d$, let $c(\mu, d)$ be the number of $\beta\in\N_0^n$ with $\wt(\beta)=d$ which have shape $\mu$.
The decomposition from Theorem~\ref{th:maindecomp} translates into the moment setting as follows.

\begin{cor}\label{thm:symmom}
For $k\in\N$, the $k$-th symmetry-adapted moment matrix $M_k^G(y)$ is
of the form
$$M_k^G(y) \ = \ \bigoplus_{\lambda\vdash n}M_{k,\lambda}^G(y).$$
Each of the blocks $M_{k,\lambda}^G(y)$ consists of $\kappa_\lambda$ rows and columns, where $\kappa_{\lambda}$ equals
$\sum_{d=0}^{k}c(\mu,d)K_{\lambda\mu}$.

\end{cor}

\begin{proof}
The distinct irreducible representations are indexed by partitions of $n$. Therefore by Remark \ref{re:size} the number of rows (and columns) 
of the block of $M_{k,\lambda}^G(y)$
corresponding to the irreducible component $S^\lambda$ equals the number of
submodules of $\R[X]_{\leq k}$ isomorphic to $S^{\lambda}$. As we have
\begin{equation}
\label{eq:rxk}
\R[X]_{\leq k} \ = \ \bigoplus_{d=0}^{k} \, \bigoplus_{\beta\in\N_0^n\, , \, \wt(\beta)=d}\R\{X^{\beta}\} \, ,
\end{equation}
this number is
$\sum_{d=0}^{k}c(\beta,d)K_{\lambda\mu(\beta)}$
by Theorem~\ref{th:maindecomp}.
\end{proof}
Although the above Corollary is stated only in terms of the symmetry adapted moment matrix, it naturally also translates to the localizing matrices. The only difference concerning the size of the localizing matrices is reflected in a slightly different calculation compared to $\kappa_\lambda$ defined above. In the case of the localizing matrix associated to a polynomial $g$ the summation will only go up to $k - \lceil \deg g / 2 \rceil$.

We obtain the following remarkable consequence.

\begin{thm} \label{th:constantasym}
For all $n\geq 2k$ the symmetry-adapted moment matrix of order $k$ has the same structure, i.e., the same number and sizes of blocks and variables. In particular, up to the computation of the block decomposition the complexity of the
question if a symmetric polynomial of degree $2k$ in $n$ variables is a sum of squares is only depending on $k$.
\end{thm}
\begin{proof}
First observe that by Remark \ref{re:size} the number of variables equals the dimension of the $\R$-vector space of symmetric polynomials of degree at most $2k$. Therefore it corresponds to the number of $n$-partitions of $2k$, which is just the number of partitions of $2k$ for all $n\geq 2k$. So we see that the number of variables does not increase in $n$ once $n\geq 2k$.

Now set $n_0 = 2k$ and let $l$ be the number of partitions of $k$, $\beta^{(1)},\ldots,\beta^{(l)} \in \N_0^{n_0}$ the distinct
exponent vectors modulo permutation with $\wt(b^{(i)}) = k$, $1 \le k \le l$,
and $\lambda^{(i)}\vdash n_0$ be the shape of $\beta^{(i)}$. The rest of the proposition follows if we can show that for every $n\geq n_0$ there exist partitions $\tilde{\lambda}^{(1)},\ldots, \tilde{\lambda}^{(m)}$ of $n$ such that $\kappa_{\tilde{\lambda}^{(i)}}=\kappa_{\lambda^{(i)}}$ for all $1\leq i\leq m$ and $\kappa_{\tilde{\lambda}}=0$ for all other $\tilde{\lambda}\vdash n$.

First note $\tilde{\beta}^{(i)}$ that are exponent vectors come from the $\beta^{(i)}$ by adding $n-n_0$ zeros.
As $n\geq n_0\geq 2k$ this implies that the possible $\tilde{\lambda}^{(i)}$ are of the form $\tilde{\lambda}^{(i)}:=(\lambda_1^{(i)}+n-n_0,\lambda_2^{(i)},\ldots,\lambda_t^{(i)})$.
Since $K_{\lambda \mu}=0$ whenever $\mu\not\unrhd\lambda$
we conclude that the possible $\tilde{\mu}$ we have to consider are of the form $\tilde{\mu}:=(\mu_1+n-n_0,\mu_2,\ldots,\mu_t)$ for one $\mu\geq\lambda^{(i)}$. But in this setting we have $K_{\lambda \mu}=K_{\tilde{\lambda} \tilde{\mu}}$ and the statement follows.
\end{proof}

\begin{example} We illustrate the techniques for a small example with
$n=3$ and $k=2$. The moment variables are indexed by partitions of the numbers $1,2,3,4$ with three parts,
i.e., $y_1,y_2,y_3,y_4,y_{11},y_{22},y_{21},y_{111},y_{211}$. The irreducible components are
indexed by the partitions $\lambda \vdash (3)$, thus $\lambda \in \{(3),(2,1),(1,1,1)\}$.
The $\beta$ we have to take into account are $(0,0,0),(1,0,0),(2,0,0),(1,1,0)$ with shape $\mu^{(1)}=(3),\mu^{(2)}=(2,1),\mu^{(3)}=(2,1),\mu^{(4)}=(2,1)$.
The semi standard generalized Young tableaux with shape $\mu$ and content
$\lambda \in \{ \lambda^{(1)}, \ldots, \lambda^{(4)}\}$ from Lemma~\ref{le:Spechtoccurrence}
are:
\begin{itemize}
\item For $\mu=(3)$:
\begin{Young}
    1 & 1 & 1 \cr
    \end{Young} \, , \begin{Young}
    1 & 1 & 2 \cr
    \end{Young} \, .
    \vspace{0.8ex}
\item For $\mu=(2,1)$:
\begin{Young}
    1& 1 \cr
    2 \cr
    \end{Young} \, .
 
\item For $\mu=(1,1,1)$ there is no generalized semi standard Young tableau corresponding to the above
$\lambda^{(i)}$.  
\end{itemize}

For $\mu = (3)$, Corollary \ref{thm:symmom} yields a $4 \times 4$-block, with basis polynomials
$$\{1,X_1+X_2+X_3,X_1^2+X_2^2+X_3^2,X_1X_2+X_1X_3+X_2X_3\} \, .$$

Thus
$$M_{(3)} \ := \ \left( \begin {array}{cccc} 1&3\,y_{1}&3\,y_{2}&3\,
y_{11}\\\noalign{\medskip}3\,y_{1}&3\,y_{2}+6\,
y_{11}&3\,y_{3}+6\,y_{21}&6\,y_{21}+3\,
y_{111}\\\noalign{\medskip}3\,y_{2}&3\,y_{3}+6\,
y_{21}&3\,y_{4}+6\,y_{22}&6\,y_{31}+3\,
y_{211}\\\noalign{\medskip}3\,y_{11}&6\,y_{21}+3\,
y_{111}&6\,y_{31}+3\,y_{211}&3\,y_{22}+6\,
y_{211}\end {array} \right) \, .$$

For $\mu=(2,1)$ we obtain a $3 \times 3$-block, with basis polynomials
\begin{eqnarray*}
&& \{X_3-X_1+X_3-X_2,(X_3-X_1)(X_3+X_1)+(X_3-X_2)(X_3+X_2),\\
&& \quad (X_3-X_1)X_2+(X_3-X_2)X_1\} \\
& = & \{2X_3-X_2-X_1,2X_3^2-X_2^2-X_1^2, -2X_1X_2+X_2X_3+X_3X_1\} \, .
\end{eqnarray*}
Thus
$$M_{(2,1)} \ = \ \left( \begin {array}{ccc} 6\,y_{2}-6\,y_{11}&6\,y_{3}-6\,y_{21}&6\,y_{21}-6\,y_{111}
\\\noalign{\medskip}6\,y_{3}-6\,y_{21}&6\,y_{4}
-6\,y_{22}&-6\,y_{211}+6\,y_{31}
\\\noalign{\medskip}6\,y_{21}-6\,y_{111}&-6\,y_{211}
+6\,y_{31}&6\,y_{22}-6\,y_{211}\end {array}
\right) \, .
$$
Hence we only have to check semidefiniteness of a $4 \times 4$ matrix and a $3 \times 3$ matrix,
instead of semidefiniteness of a single $10 \times 10$  moment matrix.

\end{example}

\begin{remark}
We remark that the techniques presented above also provide the
tools for some explicitly stated open issues  in the study of unconstrained
optimization of symmetric polynomials in Gatermann and
Parrilo \cite[p.~124]{gatermann-parrilo-2004} (who -- mentioning the lack
of explicit formulas for the isotypic components -- refer to
the study of examples and asymptotics).
\end{remark}

\subsection{Sums of squares-representations  for symmetric polynomials}\label{se:representationtheorems}
From the dual point of view,
the results presented in Section~\ref{se:momentmatrixsymmetric}
imply the following sums of squares decomposition theorem:

\begin{thm}\label{thm:darst}

Let $p \in \R[X_1, \ldots, X_n]$ be symmetric and homogeneous of degree $2d$.
If $p$ is a sum of squares then
\[
 p \ \in \ \sum_{\beta}
         \sum_{\lambda\unrhd \mu(\beta)} \sum_{T\in\mathcal{T}^0_{\lambda,\mu(\beta)}}
         \Sigma(\R\{S_{(t_{\lambda},T)}\})^2 \, ,
\]
where $t_{\lambda}$ denotes the unique $\lambda$-tableau with increasing rows and columns, $\beta$ runs over the non-negative partitions of $d$ with $n$ parts,
and $\Sigma(\R\{S_{(t_{\lambda},T)}\})^2$ denotes the sums of squares of polynomials
in the $\mathcal{S}_n$-module $\R\{S_{(t_{\lambda},T)}\}$.
\end{thm}

\begin{proof}
The statement follows from dualizing~\eqref{eq:rxk} in connection with
Theorem~\ref{th:maindecomp}.
\end{proof}

Hilbert's Theorem (see, e.g., \cite{laurent-2009}) 
characterizes the cases where the
the notion of non-negativity coincides with the sums of squares 
representability. For these cases we obtain the following corollaries which are specialized (and simplified) versions
of SOS decompositions in the symmetric case.

\begin{cor}
Let $p\in \R[X_1,X_2]$ be a symmetric homogeneous form of degree $2d$. If $p$ is non-negative then
\[
p\ \in \ \sum_{\atopfrac{\alpha_1,\alpha_2\in\N_0}{\alpha_1+\alpha_2=d}}\Sigma(\R\{X_1^{\alpha_1}X_2^{\alpha_2}+X_1^{\alpha_2}X_2^{\alpha_1}\})^2+\Sigma(\R\{X_1^{\alpha_1}X_2^{\alpha_2}-X_1^{\alpha_2}X_2^{\alpha_1}\})^2 \, .
\]

\end{cor}

\begin{cor}
Let $p \in \R[X_1, \ldots, X_n]$ be a symmetric and homogeneous quadratic form.
If $p$ is non-negative then $p$ can be written in the form

\[
 p \ = \ \alpha(X_1 + \cdots + X_n)^2 +
         \beta \sum_{i < j} (X_j - X_i)^2=(\alpha +(n-1)\beta)\sum_i X_i^2+2(\alpha-\beta)\sum_{i < j}X_iX_j
\]
with some coefficients $\alpha, \beta \ge 0$.
\end{cor}

\begin{cor}
Let $p \in \R[X_1, X_2, X_3]$ be a symmetric and homogeneous polynomial of degree~4.
If $p$ is non-negative then $p$ can be written in the form
\[
p \ = \  (\alpha+2\delta)M_{4}+(2\alpha+2\varepsilon+\gamma-\delta)M_{22}+(\beta-\omega)M_{31} +(\beta+2\gamma+2 \omega-2\varepsilon)M_{211},\\
\]
where $M_{4}=\sum_i X_i^4 $, $M_{22}=\sum_{i\neq j} X_i^2X_j^2$, $M_{31}=\sum_{i\neq j} X_i^3X_j$ and $M_{211}=\sum_{i\neq j\neq k} X_i^2 X_j X_k$,\\
such that $\alpha,\gamma,\delta,\varepsilon\geq 0$ and $\alpha\gamma\geq \beta^{2}$ and $\delta\varepsilon\geq \omega^2$.
\end{cor}

\section{Using the degree principle}\label{se:timofte}

For symmetric polynomials another possible strategy
is to use the degree principle originally introduced by Timofte \cite[Corollary ~2.1]{timofte-2003} and recently refined by Riener \cite[Theorem ~4.5]{Rie},
apparently not well-known in the optimization community. Using this principle, we
show how to reduce an
SDP-relaxation to a family of lower-dimensional relaxations.
In some situations, this strategy
might be preferable to the group-machinery 
developed in the above sections.
In the special case of symmetric polynomials of degree~4 it reduces the
non-negativity problem to an SOS problem (and thus to a semidefinite
feasibility problem), see Theorem~\ref{thm:degfour}.

\begin{prop}\cite[Theorem ~4.5]{Rie}\label{thm:degree}
Let $f,g_1,\ldots,g_m \in \R[X]$ be symmetric and $K=\{x \in \R^n \, : \, g_1(x)\geq 0,\ldots, g_m(x)\geq 0\}$. Setting
$r:=\max\big\{2,\lfloor (\deg f)/2\rfloor, \deg g_1,\ldots,$ $\deg g_m\big\}$
we have
$$\inf_{x\in K} f(x) \ = \ \inf_{x\in K\cap A_r} f(x) \, ,$$
where $A_r$ denotes the set of points in $\R^n$ with at most $r$ distinct components.
\end{prop}

For $n,r\in\N$ a vector $\omega=(\omega_1,\ldots,\omega_r)$ of positive, non-increasing integers
with $n=\omega_1+\cdots+\omega_r$ is called an $r$-\emph{partition} of $n$.
Let $\Omega$ denote the set of all possible $r$-partitions of $n$. Then for each
$\omega \in \Omega$, set $$f^{\omega}:=f(\underbrace{T_1,\ldots,T_1}_{\omega_1},\underbrace{T_2,\ldots,T_2}_{\omega_2},\ldots,\underbrace{T_{r},\ldots,T_{r}}_{\omega_{r}}) \in \R[T_1, \ldots, T_r] \, .
$$
Similarly, let
$K^{\omega}:=\{t \in\R^r \, : \, g_1^{\omega}(t) \geq0,\ldots,{g_m^{\omega}(t) \geq0\}}$.
With this notation we can transform the original optimization problem in $n$ variables
into a set of new optimization problems that involve
only $r$ variables,
\begin{equation}
\label{eq:partition}
\inf_{x\in K} f(x) \ = \
\min_{\omega\in\Omega}\inf_{t\in K^{\omega}} f^{\omega}(t) \, .
\end{equation}
Now one can apply the usual relaxation scheme to every of the above $r$-dimensional problems separately.
For each $\omega \in \Omega$ let $Q^{\omega}_k$ be the 
$k$-th relaxation $\eqref{relax}$ of $\min_{t\in K^{\omega}} f^{\omega}(t)$,
with optimal value denoted by $\inf Q_k^{\omega}$.
Putting these ideas together we obtain:

\begin{thm}\label{thm:convergence}
Let $f,g_1,\ldots,g_m\in\R[X]$ be $G$-symmetric, and let the feasible set $K$
satisfy Assumption~\ref{putinar-assumption}.
Let $r:=\max\{2,\lfloor (\deg f)/2 \rfloor,\deg g_1,\ldots,\deg g_m\},$
and $\Omega$ be the set of $r$-partitions of $n$. 
Then the sequence $(\inf_{\omega \in \Omega}(\inf Q^{\omega}_k))_{k}$ 
 converges to $\inf_{x\in K} f(x)$ for $k\rightarrow\infty$.
\end{thm}

\begin{proof}
By Theorem $\ref{thm:degree}$ there is an $r$-partition $\omega\in\Omega$ of $n$ with $\min_{x\in K} f(x)=\min_{t\in K^{\omega}} f^{\omega}(t)$. It suffices to show that $K^{\omega}$ also meets Assumption~\ref{putinar-assumption}.
Since $K$ meets Assumption~\ref{putinar-assumption}, 
there is $u\in \R[X]$ with $u=u_0+\sum_{j=1}^{m}u_jg_j$ for some sum of squares polynomials $u_0,\ldots,u_m$ such that the level set of $u$ is compact.
This representation carries over to $u^{\omega}$ which also has a compact level set.
\end{proof}
\begin{remark}
At first sight it might not look profitable to replace one initial problem by a family of new problems. However note that for fixed $r\in\N$ the number of $r$-partitions of any $n$ is bounded by $(n+r)^r$.
On the other hand a polynomial optimization problem in $n$ variables yields a moment matrix of size $O(n^{k})$ in the $k$-th relaxation step of  Lasserre's scheme.
In view of the polynomial bound (for fixed $r$) on the number of  $r$-partitions it is therefore profitable to use the degree principle-based relaxation.  
\end{remark}

The process of building the $r$-dimensional problems
can be related to breaking the symmetries as the resulting problems will in
general no longer be invariant under a symmetric group $\mathcal{S}_r$. 
However as
dimensions drop there are situations
where we will get finite convergence.

\begin{thm} \label{th:finiteconvergence}
Let $f,g_1,\ldots,g_m\in\R[X]$ be symmetric 
with $\deg g_j \le m$, $1 \le j \le m$, and $\deg f \le 2m$.
Further assume that the variety $V(g_1,\ldots,g_m)\subseteq \C^n$ 
has codimension $m$.
Then the relaxation sequence $(\inf_{\omega \in \Omega}(Q^{\omega}_k))_{k}$ 
converges to $\inf_{x\in V(g_1,\ldots, g_m)} f$ in finitely many steps,
where $\Omega$ is defined as in Theorem~\ref{thm:convergence}.
\end{thm}

\begin{proof}
It suffices to show that each of the  varieties $V^{\omega}:=V(g_{1}^{\omega},\ldots, g_{m}^{\omega})$ is zero-dimensio\-nal and then Proposition \ref{prop:zerodim} gives the announced statement. To see that these varieties contain only finitely many points we proceed as follows:

It is classically known (see for example \cite[\S{7.1}]{CLS})
that every symmetric polynomial $g$ of degree $d$ in $n$ variables can be
uniquely written as a polynomial in the first $d$ power sum polynomials
$p_1(X),\ldots,p_d(X)$, where $p_i(X)=\sum_{j=1}^n X_j^i$.

Let $\gamma_1\ldots,\gamma_m\in\R[Z_1,\ldots,Z_m]\subseteq\R[Z_1,\ldots,Z_n]$ be polynomials such that $\gamma_i(p_1(X),\ldots,$ $p_m(X))=g_i(X)$.
The surjective map 
\[
  \begin{array}{rcl}
  \pi \, : \, \C^n & \to & \C^n \, , \\
  x & \mapsto & (p_1(x),\ldots,p_n(x))
  \end{array}
\] establishes that $\C^{n}//\mathcal{S}_n$ (the so-called orbit space, cf.\ Section~\ref{se:pmirelaxations}) is in fact isomorphic to $\C^{n}$.

As the variety $V(g_1,\ldots,g_m)$ is $\mathcal{S}_n$-invariant, its image in the quotient 
$\C^{n} // \mathcal{S}_n$ is given by $\tilde{V}:=\{z\in\C^{n}\,:\,\gamma_i(z)=0,\,1\leq i\leq m\}$. Since $\mathcal{S}_n$ is a finite group the codimension of $\tilde{V}$ is also $m$. But this implies that $\tilde{V}\cap\{z\in\C^{n}\,:\,z_{m+1}=\cdots=z_n=0\}$ is zero-dimensional.
Therefore there are just finitely many $z:=(z_1,\ldots,z_m)$ with $\gamma_i(z)=0$ for $1\leq i\leq m$.

Now let $\omega=(\omega_1,\ldots,\omega_r)$ be any $r$-partition of $n$ and consider 
$V^{\omega}:=V(g_1^\omega,\ldots,g_m^\omega)\subseteq\C^{m}$.  
Setting $\tilde{p}_i:=\sum_{j=1}^m\omega_jT_j^i$, we get $g_i^{\omega}=\gamma_i(\tilde{p}_1,\ldots, \tilde{p}_m)$. For the points $y\in V^{\omega}$ we have $\tilde{p}_1(y)=z_1,\ldots,\tilde{p}_m(y)=z_m$ for one of the finitely many $z=(z_1,\ldots,z_m)$, with $\gamma_i(z)=0$ for $1\leq i\leq m$. And thus there are just finitely many points in ${V}^{\omega}$.
\end{proof}

Closely related to the question of finite convergence is the description of polynomials that are positive but not sums of squares.
By Hilbert's Theorem, every non-negative ternary quartic polynomial is a sum of squares. For quartics in more
than three variables this is not true in general, not even for symmetric polynomials (see
Example~\ref{ex:choi} below). However, for symmetric polynomials polynomials of degree~4, deciding the non-negativity
can be reduced to an SOS problem and thus to a semidefinite optimization problem.

\begin{thm}\label{thm:degfour}
Let $f\in\R[X]$ be a symmetric polynomial of degree 4,
and let $\Omega$ be the set of $2$-partitions of $n$.
Then $f$ is non-negative
if and only if for all $\omega \in \Omega$
the polynomial $f^{\omega}$ is a sum of squares.
\end{thm}
\begin{proof}
As $f$ is of degree 4, for any $\omega \in \Omega$ the polynomial
$f^{\omega}$ is of degree 4 in two variables. Hence, by
Hilbert's Theorem $f^{\omega}$ is non-negative
if and only if it is a sum of squares.
\end{proof}

\begin{ex} \label{ex:choi}
Choi and Lam \cite{choi} have shown that the homogeneous polynomial of degree~4
$$f \ = \ \sum_{i \neq j} X_i^2X_j^2+\sum_{i \neq j} X_i^2X_jX_k-4X_1X_2X_3X_4$$
in four variables is non-negative, but not a sum of squares.
By Theorem~\ref{thm:degfour}, the non-negativity of $f$ is equivalent to
the property that the following two homogeneous polynomials in two variables
are sums of squares. We find
\begin{eqnarray*}
f_1&=& X_{{2}}^{4}+4\,X_{{2}}^{2}X_{{4}}^{2}+X_{{4}}^{4}+2\,X_{{4}}^{3}X_{{2}}
=\left(X_2^2\right)^2+\left(
X_4^2+X_2X_4\right)^2+\left(\sqrt{3}X_2X_4\right)^2\, , \\
f_2&=& 4\,X_{{2}}^{4}+6\,X_{{2}}^{2}X_{{4}}^{2}-2\,X_{{2}}^{3}X_{{4}}=\left(
2 X_{2}^{2}-\frac{1}{2} X_2 X_4 \right) ^2+\left(\frac{1}{2} \sqrt
{23}X_2 X_4\right)^2 \, ,
\end{eqnarray*}
which proves that $f$ is indeed nonnegative. 
\end{ex}

\section{PMI-relaxations via the geometric quotient\label{se:pmirelaxations}}

In this section, we study another possibility to exploit symmetries. Namely, we want to exploit the fact that to any solutions of an invariant optimization problem
every point in its orbit is also optimal.   Using invariant theory and a result in real algebraic geometry it is possible to  characterize the space of all orbits. This orbit space approach leads very naturally to polynomial
matrix inequalities (PMI). The main advantage of this approach is, that in some cases, this can decrease the degrees of the polynomials strongly.
We will demonstrate this phenomenon in certain cases (such as power sum
problems) where we obtain lower bounds and sometimes even upper bounds for
a minimization problem by a very simple SDP relaxation.

\subsection{The general setup}

We consider the general $G$-invariant optimization problem~\eqref{eq:opt1},
where $G \subseteq \mathrm{GL}_n(\R)$ is a finite group. Denote
the orbit of $x \in \R^n$ by 
$G(x) := \{\sigma(x) \, : \, \sigma \in G\}$.
The union of all orbits (with the induced topology) is called the
\emph{orbit space} $\R^n // G$ of $G$. In order to characterize the
orbit space, let $\pi_1, \ldots, \pi_l$ be generators of the invariant
ring of $G$ (\emph{fundamental invariants}). The
projection
\begin{eqnarray*}
 \pi \, : \, \R^n & \to & \R^n//G \ \subseteq \ \R^l \\
          x & \mapsto & (\pi_1(x) , \ldots, \pi_l(x))
\end{eqnarray*}
defines an embedding of the orbit space into $\R^l$. In contrast to the complex case (see \cite{CLS}) this map is not surjective in general.  We highlight this phenomenon with the following example:
\begin{ex} \label{eq:orbitspacesymm}
Let $G=D_4$ be the dihedral group acting on $\R^2$. Fundamental invariants that generate $\C[X,Y]^{D_4}$ are given by $f_1=x^2+y^2$ and $f_2=x^2y^2$ (for general
methods to compute fundamental invariants we refer to \cite{DerksenKemper}).
As $f_1$ and $f_2$ are in fact algebraically independent, we find that $\C^{n}//D_4\simeq\C^2$. 
In the complex setting every solution to the linear system $z_1=\alpha_1, z_2=\alpha_2$ for some $(\alpha_1,\alpha_2)\in \C^2$ will give rise to an orbit of solutions of $f_1(x,y)=\alpha_1, f_2(x,y)=\alpha_2$. But since for optimization purposes we are interested in real solutions we have to restrict the map $\pi$ to $\R^2$. Obviously, the image $\pi(\R^2)$ is contained in $\R^2$. On the other hand, as $f_1(x,y)\geq 0$ for all $(x,y)\in \R^2$, we have $\pi^{-1}((-1,0))\not\in\R^2$. Hence, 
the restricted map  $\pi_{|{\R^2}}$ is not surjective. 
\end{ex} \noindent 

Therefore, in order to describe the image of $\R^n$ under $\pi$ we need to add further constraints. 
In Example~\ref{eq:orbitspacesymm} for instance, the property $f_1(x,y)\geq 0$ for all $(x,y)\in\R^n$ implies $\pi(\R^2)\subseteq \left\{(z_1,z_2)\in\R^2\,:\,z_1\geq 0\right\}$. Thus it seems promising to add such positivity constraints to characterize the image $\pi(\R^n)$ as a semi algebraic subset of $\R^n$. This is indeed possible and the characterization has been done by Procesi and Schwarz, who have determined polynomial inequalities which have to be taken additionally into account in order to characterize
the embedding of $\R^n // G$ into the coordinate variety of $\R[X]^G$
of $G$ (see also Br\"ocker \cite{broecker}). We outline this briefly:

First note that there exists a $G$-invariant inner product $\langle \cdot , \cdot \rangle$ on
$\R[X]$.
For a polynomial $p$ the differential $dp$ is defined by
$dp = \sum_{j=1}^n \frac{\partial p}{\partial x_j} d x_j$.
Then carrying over the inner product to the differentials yields
$\langle dp, dq \rangle = \sum_{j=1}^n \langle \frac{\partial p}{\partial x_j} ,
 \frac{\partial q}{\partial x_j} \rangle$ .
The inner products $\langle d\pi_i, d \pi_j \rangle$
$(i,j \in \{1, \ldots, l\})$ are $G$-invariant. Hence
entry of the symmetric matrix
\begin{equation}
  \label{eq:j}
 J \ = \ ( \langle d\pi_i, d \pi_j\rangle)_{1 \le i,j \le l}
\end{equation}
is $G$-invariant and can therefore be expressed in terms of 
$\pi_1, \ldots, \pi_l$.

\begin{prop}[Procesi, Schwarz \cite{proschwa}] \label{pr:procesischwarz}
Let $G \subseteq \mathrm{GL}_n(\R)$ be a finite group and
$\pi = (\pi_1, \ldots, \pi_l)$ be fundamental invariants of $G$.
Then the orbit space is given by polynomial inequalities,
\[
 \R^n // G \ = \ \pi(\R^n) \ = \ \{z \in \R^l \, : \, J(z) \succeq 0 \, , \,
    z \in V(I) \} \, ,
\]
where $I \subseteq \R[z_1, \ldots, z_l]$ is the ideal of algebraic
relations among $\pi_1, \ldots, \pi_l$.
\end{prop}
 
\begin{ex} Continuing Example~\ref{eq:orbitspacesymm}, we have
$\frac{\partial f_1}{\partial x}=2x$, $\frac{\partial f_1}{\partial y}=2y$, $\frac{\partial f_2}{\partial x}=2 xy^2$, $\frac{\partial f_2}{\partial y}=2x^2y$.
Expressed in the original variables $x,y$, and in the fundamental invariants $f_1,f_2$,
respectively, the matrix $J$ from~\eqref{eq:j} is
$$J \ = \ \begin{pmatrix}
4(x^2+y^2)&8x^2y^2\\
8x^2y^2&4(x^2y^4+y^2x^4)
\end{pmatrix}
\ = \ 
\begin{pmatrix}
4f_1&8f_2\\
8f_2&4f_1f_2
\end{pmatrix} \, .$$
With the principle minors of $J$ (which are
$4f_1$, $4f_1f_2$ and $4f_1 \cdot 4f_1f_2- (8f_2)^2$), the orbit space is given by
$$\R^2//D_4 \ = \ \left\{(z_1,z_2)\in\R^2\,:\,4z_1\geq 0,\, 4z_1z_2\geq 0,\, 16z_1^2z_2- (8z_2)^2\geq 0\right\}.$$
\end{ex} 

Let $\tilde{f}$ and $\tilde{g}_1,\ldots,\tilde{g}_m$ be the expressions for $f$ and $g_1,\ldots, g_m$ in the fundamental 
invariants. By Proposition~\ref{pr:procesischwarz},
the $G$-symmetric optimization problem~\eqref{eq:opt1}
can be equivalently expressed in the orbit space:
\begin{equation}
 \label{eq:optinorbitspace}
 \begin{array}{rrcll}
 \multicolumn{4}{l}{\inf \tilde{f}(z)} & \\
 \text{s.t.} & z & \in & V(I) \, , & \\
 & \tilde{g}_j(z) & \geq & 0 \, , & \quad 1 \le j \le m \,  , \\
 & J(z) & \succeq & 0 \, . &
 \end{array}
\end{equation}
This is a PMI (as introduced in Section~\ref{se:pmi}) and one can use the techniques introduced there to derive an SDP relaxation scheme. 
Let $s_1(z),\ldots,s_r(z)$ be generators for 
the algebraic relations among $\pi_1,\ldots, \pi_l$. 
Then~\eqref{relax2} yields the hierarchy 
$Q_k^{//}$ of SDP relaxations

\begin{equation}\label{qutient}
 Q_k^{//}:\quad\begin{array}{rcl}
  \multicolumn{3}{l}{\inf_{y} L(f)} \\
 M_k(y) & \succeq & 0 \, , \\
 M_{k-d}(J \, y) & \succeq & 0 \, ,\\
 M_{k - \lceil \deg\tilde{g}_j / 2 \rceil}(\tilde{g}_j \, y) & \succeq & 0 \, , 
  \quad 1 \le j \le m \, , \\
  M_{k - \lceil \deg s_j / 2 \rceil}(s_j \, y) & =& 0 \, ,
  \quad 1 \le j \le r \, ,
 \end{array}
\end{equation}
where $k \ge k_0 := \max \{  \lceil \deg \tilde{f} / 2 \rceil,
\lceil \deg \tilde{g} / 2 \rceil,
\lceil \deg s_j / 2 \rceil,
d \}$
and $d = \max\{ \lceil \deg J_{ij}(Z)/2 \rceil\}$.

\begin{thm}
Let $f, g_1,\ldots,g_m$ be $G$-invariant. If the PMI in \eqref{qutient} meets Assumption~\ref{putinar-assumption2} then the sequence $(\inf Q^{//}_k)_{k\geq k_0}$ is monotonically non-decreasing and
converges to $f^*$.

\end{thm}
\begin{proof}
By Proposition \ref{pr:procesischwarz} the problem
described by $f$ and $g_1,\ldots,g_m$ is equivalent to \eqref{eq:optinorbitspace}.
Now we can conclude with Proposition~\ref{prop:Las_PIM}.
\end{proof}

\begin{remark}
It would be very interesting to characterize the situations
where \eqref{pr:procesischwarz} meets condition \ref{putinar-assumption2} in terms of the original set $K$.
\end{remark}

\subsection{Lower and upper bounds for power sum problems\label{se:powersums}}

For constrained polynomial optimization problems described
by power sums, the PMI become particularly simple. We will use the first relaxation of the sequence \eqref{qutient} in order to derive bounds for a particular class of problems.

Let $n,m,q\in\N$ with $q\geq m$, $m\leq n+1$, and given some vector
$\gamma\in\R^{m-1}$, consider the symmetric global optimization problem

\begin{equation}
 \label{eq:simplexex}
\p_{nmq}:\quad
\min \:\sum \limits_{i=1}^{n}  x_i^q \qquad
 \text{s.t.}\quad  \sum \limits_{i=1}^{n}  x_i^j \ = \ \gamma_k \, , \quad j=1,\ldots,m-1 \, ,
 \end{equation}
with optimal value denoted $\min\p_{nmq}$. Here, we provide upper and lower bounds for $\p_{nmq}$.

Choose the fundamental invariants $\pi_j = \frac{1}{j} s_j$ ($1 \le j \le n$)
where $s_j := \sum_{i=1}^n x_i^j$ denotes the power sum of order $j$.
Then the matrix $J(z)$ specializes to the Hankel matrix $H_n(s) = (s_{i+j-2})_{1 \le i,j \le n}$.

We can exploit the double occurrence of power sums: within the
optimization problem and within the Hankel matrix.
Namely, for $m\leq n+1$ and $m \le q\leq 2n-2$, consider the
following semidefinite optimization problem
\begin{equation}
\label{relax-sdp}
\l_{nmq}\,=\,\min_{s}\:\{\:s_q \; \vert \; H_n(s)\succeq0\,;\:
s_0=n\,;\:s_j=\gamma_j,\quad j=1,\ldots,m-1\} \, .
\end{equation}

\begin{thm}
\label{lower1}
Let $n,m,q\in\N$ with $m\leq n+1$, $m \le q\leq 2n-2$, and let
$\p_{nmq}$ be as in (\ref{eq:simplexex}). Then one obtains the following lower bounds on $\min\p_{nmq}$.

{\rm (a)} $\min\p_{nmq}\geq \l_{nmq}$.

{\rm (b)} If $q=m=2r$ for some $r$, write
\[  H_{r+1}(s)\ = \   
    \left( \begin{array}{c|c}
   H_{r}(\gamma) & u_r(\gamma) \\ \hline \\ [-2ex]
   u_r^T(\gamma) &  s_{2r}
 \end{array} \right);\quad u_r(\gamma)^T=(\gamma_{r},\ldots,\gamma_{2r-1}) \, , \]
with $\gamma_0=n$. Then $\min\p_{nmq} \ge u_r(\gamma)^T H_{r}(\gamma) u_r(\gamma)$.
\end{thm}

\begin{proof}
(a) Consider the equivalent formulation to~\eqref{eq:simplexex} in the form \eqref{eq:optinorbitspace}.  Since $J(z)$ is a Hankel matrix, every solution to this PMI is feasible for \eqref{relax-sdp}.

(b) In case $q=m=2r<2n$, we observe $r < n$ and
\[  H_n(s)\ = \   
    \left( \begin{array}{c|c}
   H_{r+1}(s) &U(s) \\ \hline \\ [-2ex]
   U^T(s) & V(s)
 \end{array} \right),\]
 for some suitable
 (possibly empty) matrices $U(s)\in\R^{(r+1)\times (n-r-1)}$,
 $V(s)\in \R^{(n-r-1)\times (n-r-1)}$. Therefore, $H_n(s)\succeq0$ implies $H_{r+1}(s)\succeq0$,
 and the final result follows from Schur's complement applied
 to the Hankel matrix $H_{r+1}(s)$.
\end{proof}

In certain cases, we can complement this lower bound for problem~\eqref{eq:opt1} by
an upper bound. The idea is to consider potential solutions $x\in\R^n$ of
$\p_{nmq}$ with at most $m$ non-zero components.

Consider the monic polynomial $p =X^m + \sum_{k = 0}^{m-1} p_j X^j \in \R[X]$,
and let $x_1, \ldots, x_m$ be the $m$ roots (counting multiplicities)
of $p$.
A necessary and sufficient
condition for all roots of $p$ to be real is that $H_m(s)\succeq0$, where
$H_m(s)$ is the Hankel matrix with $s_0=m$.

When $q\leq 2m-2$, we investigate the following SDP problem
\begin{equation}
\label{newrelax-sdp}
\u_{nmq}\,=\,\min_{s}\:\{\:s_q \; \vert \; H_m(s)\succeq0\,;\:
s_0=m\,;\:s_j=\gamma_j,\quad j=1,\ldots,m-1\},
\end{equation}
which the same as (\ref{relax-sdp}) except that we now have a
Hankel matrix $H_m(s)$ of dimension $m$ instead of $H_n(s)$ of dimension $n$.

It is well known that the Newton sums
$s_k \ = \ \sum_{j=1}^m X_j^k$, $k \ge 0$,
of $p$ are known polynomials in its coefficients $\{p_j\}$, and
conversely, the coefficients $p_j$ of $p$ are polynomials in the
$s_j$'s. i.e., we can write $s_j \ = \ P_j(p_0, \ldots, p_{m-1}) $, $j \ge 0$
for some polynomials $P_j\in\R[p_0,\ldots,p_{m-1}]$. 
In fact, the $s_i$'s and the $p_j$'s are related by Newton's identities,
 \begin{eqnarray*}
   s_k + p_{m-1} s_{k-1} + \cdots + p_0 s_{k-m} & = & 0 \qquad (k \geq m) \, , \\
   s_k + p_{m-1} s_{k-1} + \cdots + p_{m-k+1} s_1 & = & - k p_{m-k}
   \qquad ( 1 \le k < m) \, .
 \end{eqnarray*}

If one knows $s_j$ for all $j=1,\ldots, m-1$,
then one may compute the $p_j$'s for all
$j=1,\ldots, m-1$, and therefore, we can choose as unknown of our problem the variable
$p_0$ (the only (constant) coefficient of $p$ that we do not know), and
write
\[
 s_j \ = \ P_j(p_0, \ldots, p_{m-1}) \ = \ Q_j(p_0) \, , \qquad
 j = m, m+1, \ldots
\]
for some known polynomials $Q_j\in\R[p_0]$. We claim that $Q_j$ is
affine whenever $j \le 2m-1$.
Indeed, this follows from
\begin{eqnarray*}
 s_m & = & - s_0 p_0 - s_1 p_1 - \cdots - s_{m-1} p_{m-1} \, , \\
 s_{m+1} & = & - s_1 p_0 - \cdots - s_{m-1} p_{m-2} - s_{m} p_{m-1} \\
       & = & -s_1 p_0 - \cdots - s_{m-1} p_{m-2} +
       p_{m-1}(s_0 p_0 + s_1 p_1 + \cdots + s_{m-1} p_{m-1} ) \\
       & = & - p_0 (s_1 - s_0 p_{m-1}) - p_1 (s_2 - p_{m-1} s_1) - \cdots
           - p_{m-1} (s_m - p_{m-1} s_{m-1}) \, , \\
 s_{m+2} & = & - s_2 p_0 - \cdots - s_{m-1} p_{m-3} - s_m p_{m-2} - s_{m+1} p_{m-1} \\
        & = & - p_0 (s_2 - p_{m-2} s_0 + p_{m-1} s_1 - s_0 p^2_{m-1}) - \cdots \, , \\
 s_{m+3} & = & - p_0 (s_3 - s_0 p_{m-3} + \cdots) - \cdots
\end{eqnarray*}

Therefore, with $q\leq 2m-2$, the SDP problem (\ref{newrelax-sdp}) reads
\begin{equation}
\label{newupperbound}
\u_{nmq}\,=\,  \min_{p_0}\: \{ \: Q_q(p_0) \, : \, H_{m}(s) \succeq 0 \} \, ,
\end{equation}
where $s_0=m$ and all the entries $s_j$ of $H_{m}(s)$ are replaced by their affine expression
$Q_j(p_0)$ whenever $m \le j \le 2m-2$.
This is an SDP with the single variable $p_0$ only.
\begin{thm}
\label{upper}
Let $n,m,q\in\N$ with $m\leq n$ and $q\leq 2m-2$. Let
$\p_{nmq}$ be as in (\ref{eq:simplexex}) and let
$\u_{nmq}$ be as in (\ref{newupperbound}). Then $\min \p_{nmq}\,\leq\,\u_{nmq}$.

In addition, if $\p_{nmq}$ has an optimal solution $x^*\in\R^n$ with at most
$m$ non-zero entries, then $\min\p_{nmq}=\u_{nmq}$ and so
$\p_{nmq}$ has the equivalent convex formulation (\ref{newupperbound}).
\end{thm}
\begin{proof}
Let $p_0$ be an optimal solution of the SDP (\ref{newupperbound}), and consider
the monic polynomial $p\in\R[X]$ of degree $m$
which satisfies the Newton identities with
$s_j=\gamma_j$, $j=1,\ldots,m-1$. The vector $x=(x_1,\ldots,x_m)$ of all its roots
(counting multiplicities) is real because $H_m(s)\succeq0$, i.e.,
its Hankel matrix $H_m(s)$ formed with its Newton sums
$s_j$, $j=1,\ldots,2m-2$ (and $s_0=m$), is positive semidefinite. Let
$x^*=(x,0,\ldots, 0)\in\R^n$. By definition of the Newton sums of $p$, one has
$\sum_{i=1}^n (x^*_i)^k\,=\,\sum_{i=1}^m x^k_i\,=\,\gamma_k$, $k=1,\ldots, m-1$,
which shows that $x^*$ is feasible for $\p_{nmq}$. Therefore,
$\u_{nmq}=s_q\geq\min\p_{nmq}$, the desired result.
\end{proof}

\medskip

\noindent
{\bf Acknowledgement.} We are very grateful for the comments of two 
anonymous referees which helped us a lot to improve the presentation.

\end{document}